\documentclass[10pt]{article}
\usepackage{latexsym}

\usepackage[top=1in, bottom=1in, left=1in, right=1in]{geometry}

\usepackage{adjustbox}

\usepackage{amssymb}

\usepackage{graphicx} 
\usepackage{color} 
\usepackage{amsmath, amsthm, amssymb, appendix}
\usepackage{enumerate}
\usepackage{float}
\usepackage{url}
\usepackage{stackrel}
\usepackage{mathrsfs,dsfont}
\usepackage[labelfont=bf]{caption}
\usepackage{subcaption}
\usepackage{wrapfig}
\usepackage{bbm}
\usepackage{bm}
\usepackage{epigraph}
\usepackage{physics}
\usepackage{multirow}
\usepackage{hyperref}
\usepackage{pdflscape}
\usepackage{tikz}
\usetikzlibrary{shapes,arrows,automata,graphs,decorations.pathmorphing,decorations.markings}
\usepackage{tikz-cd}
\usepackage{fancybox}
\usepackage{booktabs}

\usepackage{mdwlist}

\usepackage{pifont}

\newtheorem{theorem}{Theorem}[section]

\newtheorem{proposition}[theorem]{Proposition}
\newtheorem{lemma}[theorem]{Lemma}

\newtheorem{remark}[theorem]{Remark}

\theoremstyle{definition}

\theoremstyle{definition}
\newtheorem{definition}[theorem]{Definition}
\theoremstyle{definition}

\newcommand{\been}{\begin{enumerate}}
\newcommand{\enen}{\end{enumerate}}
\newcommand{\beit}{\begin{itemize}}
\newcommand{\enit}{\end{itemize}}

\def\EE{\mathcal{E}}

\def\cre
\def\cgn{\color{olive}}
\def\cbl
\def\cma
\def\ccy{\color{cyan}}

\usepackage{tikz}    
\usetikzlibrary{shapes,automata,positioning,arrows,fit}
\usepackage{mathtools}  
\usetikzlibrary{snakes}

\usepackage{caption}
\usepackage{tikz-cd}
\usetikzlibrary{babel}

\usepackage{tkz-euclide}

\usepackage{tikz-3dplot}

\usepackage{relsize}

 \tikzset{every node/.style={auto}}
 \tikzset{every state/.style={rectangle, minimum size=0pt, draw=none, font=\Large}}
  \tikzset{bend angle=7}
  
  \usepackage{array}
  
  \usepackage{chemfig}

\usepackage{authblk}

\makeatletter
\newcommand{\xrightleftarrows}[2]{
  \mathrel{\mathop{
    \vcenter{\offinterlineskip\m@th
      \ialign{\hfil##\hfil\cr
        \hphantom{$\scriptstyle\mspace{8mu}{#1}\mspace{8mu}$}\cr
        \rightarrowfill\cr
        \vrule height0pt width 2em\cr
        \leftarrowfill\cr
        \hphantom{$\scriptstyle\mspace{8mu}{#2}\mspace{8mu}$}\cr
        \noalign{\kern-0.3ex}
      }
    }
  }\limits^{#1}_{#2}}
}
\newcommand{\xleftrightarrows}[2]{
  \mathrel{\mathop{
    \vcenter{\offinterlineskip\m@th
      \ialign{\hfil##\hfil\cr
        \hphantom{$\scriptstyle\mspace{8mu}{#1}\mspace{8mu}$}\cr
        \leftarrowfill\cr
        \vrule height0pt width 2em\cr
        \rightarrowfill\cr
        \hphantom{$\scriptstyle\mspace{8mu}{#2}\mspace{8mu}$}\cr
        \noalign{\kern-0.3ex}
      }
    }
  }\limits^{#1}_{#2}}
}
\makeatother

\title{Bistability, Absolute Concentration Robustness, and Hysteresis \\ in Dual-Site Futile Cycles with Bifunctional Enzymes}
\author{Badal Joshi\thanks{Department of Mathematics, California State University San Marcos. \texttt{bjoshi@csusm.edu}} \hspace{1em}
Tung D. Nguyen\thanks{Department of Mathematics, University of California Los Angeles. \texttt{tungdnguyen@math.ucla.edu}} \hspace{1em} Matthew D. Johnston\thanks{Department of Mathematics and Computer Science, Lawrence Technological University. \texttt{mjohnsto1@ltu.edu}}}

\begin{document}
\maketitle

\begin{abstract}

Bifunctional enzymes, which catalyze both the forward and reverse steps of a substrate modification reaction, arise naturally in bacterial two-component signaling systems and metabolic regulation. Beyond their well-known role in conferring absolute concentration robustness (ACR) on substrate species, bifunctional enzymes profoundly shape the dynamical landscape of the networks in which they appear. We study a class of dual-site futile cycles in which the reverse modification steps are carried out by bifunctional enzyme-substrate compounds, and provide a complete mathematical analysis of all four such networks, characterizing the existence, number, and stability of steady states, as well as the bifurcation structure as total substrate is varied. All four networks admit boundary steady states, in contrast to the non-bifunctional case. The networks differ in the number and stability of boundary steady states, in the maximum number of positive steady states (ranging from two to four), and in whether bistability is present. In two networks, a transcritical bifurcation connects the boundary and positive steady state branches; in one case this is a backward bifurcation, producing hysteresis. Perhaps the most striking phenomenon occurs in one of the four networks, which simultaneously exhibits bistability and ACR in the final modification state, where the system can settle into either of two stable steady states with different intermediate concentrations yet identical final product concentration.

\textbf{Relevance to Life Sciences.}
Reversible protein phosphorylation at multiple sites is a ubiquitous mechanism for signal transduction in living cells. Bifunctional enzymes, which both phosphorylate and dephosphorylate their substrates, are found in bacterial signaling systems such as EnvZ/OmpR and in metabolic regulation such as the IDHKP-IDH system of \textit{Escherichia coli}. Our results reveal that the specific choice of which intermediate compound acts as the bifunctional enzyme 
has profound consequences for the robustness and dynamical repertoire of the signaling system. Networks in which the bifunctional enzyme is the compound formed in the final phosphorylation step exhibit ACR in the final modification state, meaning the output signal is completely insensitive to changes in total substrate or enzyme levels. In the network without individual ACR, the product of two intermediate concentrations is nonetheless invariant, and for specific parameter values this leads to a counterintuitive negative dose-response: increasing total substrate decreases the final modification product. This is in contrast to the non-bifunctional case, where a biphasic dose-response has been shown to arise.

\textbf{Mathematical Content.}
We model four dual-site substrate modification networks under mass-action kinetics with the Henri-Michaelis-Menten mechanism. 
We derive a unified steady-state parameterization in which all steady-state coordinates are expressed in terms of a single variable, reducing the problem of finding steady states to finding positive roots of an explicitly computable single-variable polynomial. Descartes' rule of signs gives  upper bounds on the number of positive steady states, and explicit parameter choices show that the stated network-specific maxima are attained. Stability of boundary steady states is analyzed using the next-generation matrix method, adapted from mathematical epidemiology, which provides a spectral radius threshold criterion for stability. Stability of positive steady states is determined from eigenvalues of the Jacobian matrix. ACR properties are  recalled from prior work using the substrate hypergraph framework of Joshi and Nguyen, which identifies robustness directly from network structure without requiring analysis of the ODE system, and applies to any detailed mechanistic model sharing the same substrate hypergraph.

\end{abstract}

\section{Introduction}

\begin{figure}[t]
    \centering
    \begin{subfigure}{0.45\textwidth}
        \centering
        \begin{tikzpicture}[baseline={(current bounding box.center)},node distance=2cm, on grid]
            \node (S0)  {$S_0$};
            \node (S1) [right=3cm of S0] {$S_1$};
            \node (S2) [right=3cm of S1] {$S_2$};
            \draw[->,line width=1.25] (S0) .. controls +(0.5,1.23) and +(-0.5,1.23) .. node[above] {$(E, C_1)$} (S1);
            \draw[->,line width=1.25] (S1) .. controls +(-0.5,-1.23) and +(0.5,-1.23) .. node[below] {$(\EE_1, C_4)$} (S0);
            \draw[->,line width=1.25] (S1) .. controls +(0.5,1.23) and +(-0.5,1.23) .. node[above] {$(E, C_2)$} (S2);
            \draw[->,line width=1.25] (S2) .. controls +(-0.5,-1.23) and +(0.5,-1.23) .. node[below] {$(\EE_2, C_3)$} (S1);
        \end{tikzpicture}
        \caption{Substrate hypergraph}
        \label{fig:sub2_state}
    \end{subfigure}
    \hfill
    \begin{subfigure}{0.50\textwidth}
        \centering
        {\footnotesize
        \begin{gather*}
        S_0 + E \;\xrightleftarrows{k_1^+}{k_1^-}\; C_1 \;\xrightarrow{k_1}\; S_1 + E
        \qquad
        S_1 + E \;\xrightleftarrows{k_2^+}{k_2^-}\; C_2 \;\xrightarrow{k_2}\; S_2 + E \\[6pt]
        S_0 + \EE_1 \;\xleftarrow{k_4}\; C_4 \;\xleftrightarrows{k_4^+}{k_4^-}\; S_1 + \EE_1
        \qquad
        S_1 + \EE_2 \;\xleftarrow{k_3}\; C_3 \;\xleftrightarrows{k_3^+}{k_3^-}\; S_2 + \EE_2
        \end{gather*}
        }
        \caption{Detailed model}
        \label{fig:sub1_reaction}
    \end{subfigure}
    \caption{Dual-site futile cycle with bifunctional enzymes $\EE_1$ and $\EE_2$. We consider all cases where $\EE_i \in \{ C_1, C_2\}$ for $i=1, 2$. The forward steps are catalyzed by a common enzyme $E$ and the reverse steps by one or both of the intermediate compounds $C_1$ and $C_2$. We denote the reaction network above with the notation $(E,E,\EE_1,\EE_2)$.}
    \label{fig:dualsite}
\end{figure}
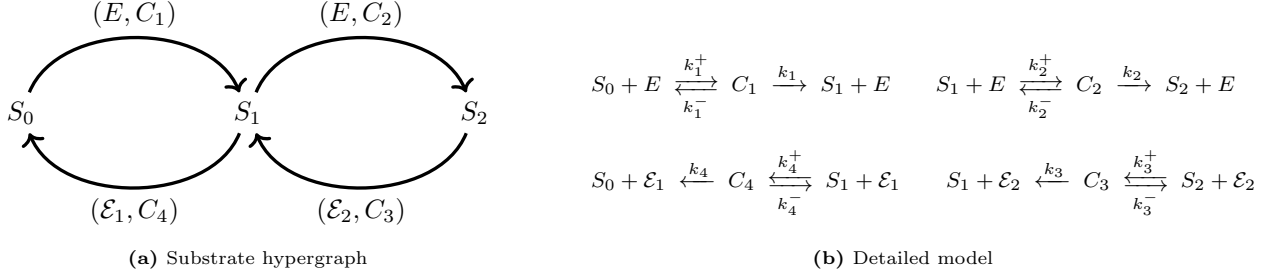

Reversible covalent modification of proteins (most commonly phosphorylation and dephosphorylation) is a fundamental mechanism by which cells transduce, amplify, and respond to signals \cite{shinar2009robustness}. In a \emph{futile cycle} (also called a substrate cycle), a substrate is repeatedly modified and demodified by opposing enzymatic reactions that together consume energy without net chemical change. Despite this apparent futility, such cycles serve critical regulatory functions: they can generate ultrasensitive, switch-like responses \cite{goldbeter1981amplified}, maintain concentrations at precise levels insensitive to fluctuations \cite{shinar2010structural}, and produce bistability that underlies cellular memory and decision-making \cite{ferrell1996tripping}.

A classical single-site futile cycle involves a substrate that cycles between two forms (e.g., unphosphorylated $S_0$ and phosphorylated $S_1$), with forward and reverse reactions catalyzed by distinct enzymes. The mathematical analysis of such systems under mass-action kinetics is relatively well understood: the Henri-Michaelis-Menten mechanism yields a unique positive steady state for each choice of parameters \cite{feinberg2019foundations}. The situation becomes considerably richer at multiple phosphorylation sites, where the substrate can exist in several modification states. Dual-site phosphorylation, in which a substrate cycles through three states $S_0$, $S_1$, $S_2$, has attracted particular attention due to its prevalence in cell signaling and its capacity for complex dynamics. Under distributive (non-processive) mechanisms, dual-site phosphorylation cycles can admit multiple positive steady states \cite{markevich2004signaling, wang2008singularly}, and the number and stability of these steady states depend sensitively on the kinetic mechanism and parameter values.

A biologically important variant that has received comparatively less mathematical attention is the case of \emph{bifunctional enzymes}, which carry out both the forward and reverse reactions on the substrate. In a standard futile cycle, separate enzymes catalyze phosphorylation and dephosphorylation. In a bifunctional system, a single enzyme species performs both functions. This architecture arises naturally in prokaryotic two-component signaling systems, such as the EnvZ/OmpR system in \textit{Escherichia coli} \cite{russo1993essential, hsing1998mutations, batchelor2003robustness}, and in metabolic regulation; a prominent example is the IDHKP-IDH system, in which the bifunctional enzyme isocitrate dehydrogenase kinase/phosphatase (IDHKP) both phosphorylates and dephosphorylates isocitrate dehydrogenase (IDH) \cite{shinar2009robustness, dexter2013dimerization}. The single-site futile cycle with a bifunctional enzyme has been studied in detail: Shinar and Feinberg showed that it exhibits \emph{absolute concentration robustness} (ACR) in the substrate species \cite{shinar2010structural}, meaning the steady-state concentration of the substrate is determined entirely by the enzyme kinetics and is independent of total enzyme or substrate levels. This result is striking from both a mathematical and a biological perspective: it implies that the system output is buffered against perturbations that change total protein levels, a property that may underlie the remarkable robustness observed in bacterial signaling \cite{shinar2009robustness, batchelor2003robustness}.

The concept of ACR, introduced by Shinar and Feinberg \cite{shinar2010structural}, identifies a network-structural condition that guarantees robustness of a specific species concentration at steady state. Extensions and variations of ACR have since been studied extensively, including connections to stochastic dynamics \cite{anderson2014stochastic}, dynamic ACR \cite{joshi2022foundations, joshi2023reaction}, and algebraic characterizations \cite{puente2025absolute}. More recently, Joshi and Nguyen developed a \emph{substrate hypergraph} approach that allows ACR to be identified directly from the network structure of the substrate hypergraph (a coarse-grained description of which substrates are transformed in each reaction, shared by many different detailed models and ODE systems) without specifying the precise reaction steps or enzyme-substrate interactions \cite{joshi2026bifunctional}.

Stability of boundary steady states, which are  steady states at which some species concentrations are zero, presents a distinct technical challenge. Standard linearization-based stability criteria apply to interior  steady states, but boundary  steady states require more care because the relevant perturbations must remain in the non-negative orthant. The \emph{next-generation matrix} (NGM) method, originally developed in mathematical epidemiology to compute the basic reproduction number $\mathcal{R}_0$ \cite{van2002reproduction, van2008further}, has recently been adapted to establish stability of boundary steady states in mass-action systems \cite{avram2024advancing, avram2024stability, johnston2026boundary}. This approach provides a threshold quantity $\rho(FV^{-1})$, the spectral radius of a non-negative matrix constructed from the linearization at the boundary  steady state, that determines stability: the boundary  steady state is stable when $\rho(FV^{-1}) < 1$ and unstable when $\rho(FV^{-1}) > 1$.

The dynamics of \emph{dual-site} futile cycles with bifunctional enzymes has not been systematically analyzed. The present paper undertakes this analysis for all four combinatorial choices of bifunctional enzyme in a  distributive dual-site cycle: $(E,E,C_2,C_1)$, $(E,E,C_1,C_1)$, $(E,E,C_1,C_2)$, and $(E,E,C_2,C_2)$, where the notation specifies which enzyme catalyzes each of the four elementary reactions (see Figure \ref{fig:dualsite}). The ACR properties of these networks were identified in \cite{joshi2024bifunctional, joshi2026bifunctional}; the present paper provides the first systematic analysis of their full dynamical behavior, including boundary steady state existence and stability, positive steady state capacity and multiplicity, bifurcation structure, and multistability.

Our analysis reveals a striking diversity of dynamical behaviors across the four networks, all arising from the same underlying enzymatic architecture with only the assignment of bifunctional roles changed. The networks differ in the number and stability of boundary steady states (one or two, stable or unstable), the maximum number of positive steady states (two, three, or four), the nature of the bifurcation connecting boundary and positive steady state branches (transcritical, saddle-node, or backward bifurcation), and whether bistability is present. The ACR properties further differentiate the networks: one network exhibits ACR only in the final product $S_2$, one exhibits ACR only in the intermediate species $S_1$, one has ACR in both $S_1$ and $S_2$, and one has no ACR. In the network without ACR, the product $s_1 s_2$ is nonetheless invariant across positive steady states, a weaker form of robustness that constrains the joint variation of $s_1$ and $s_2$; for the specific parameter values studied, this forces a negative dose-response in $s_2$ whenever $s_1$ increases with $T_s$. More broadly,  for the parameter regimes studied here, we do not observe a positive dose-response of the final modification product to total substrate, in contrast to the non-bifunctional case, where a biphasic dose-response (first increasing, then decreasing) has been shown to arise from enzyme sequestration effects \cite{suwanmajo2013biphasic}.  We discuss these dose-response properties further in Section~\ref{sec:discussion}. This combination of robustness and complexity makes the dual-site bifunctional enzyme system a rich object of study both mathematically and biologically.

The paper is organized as follows. Section~\ref{sec:background} presents the mathematical background, including mass-action kinetics, stoichiometric compatibility classes, steady state definitions, boundary steady state stability via the next-generation matrix method, and absolute concentration robustness. Section~\ref{sec:SSP_1site} analyzes the single-site network $(E,C_1)$, which serves as motivation and a building block for the dual-site analysis. Section~\ref{sec:dualsite} contains the main results for the four dual-site networks. Section~\ref{sec:discussion} discusses the broader implications of our findings, comparing the networks to each other and to the non-bifunctional case. An appendix provides mathematical background and a Further Reading list for non-specialist readers. Detailed algebraic computations and proofs are collected in the Supplementary Material.

\section{Background}
\label{sec:background}

In this section we introduce the mathematical framework of chemical reaction networks and mass-action kinetics, and collect the definitions and results that will be used throughout the paper. Throughout, capital letters denote chemical species and the corresponding lowercase letters denote their concentrations (for example, $E$ and $e$). For the futile-cycle models below, $T_e$ and $T_s$ denote the conserved total amounts of enzyme and substrate, respectively.

\subsection{Reaction Networks and Mass-Action Systems}

A \emph{chemical reaction network} consists of a set of species $\{X_1, \ldots, X_n\}$ and a set of reactions, each of the form
\[
\sum_{i=1}^n \alpha_{ij} X_i \;\longrightarrow\; \sum_{i=1}^n \beta_{ij} X_i, \quad j = 1, \ldots, r,
\]
where the non-negative integer vectors $\alpha_j = (\alpha_{1j}, \ldots, \alpha_{nj})$ and $\beta_j = (\beta_{1j}, \ldots, \beta_{nj})$ are the \emph{source} and \emph{product complexes} of reaction $j$, respectively. Under \emph{mass-action kinetics}, the rate of reaction $j$ is assumed to be proportional to the product of the concentrations of the reacting species, with proportionality constant $k_j > 0$ called the \emph{rate constant}. Denoting by $x = (x_1, \ldots, x_n) \in \mathbb{R}^n_{\geq 0}$ the vector of species concentrations, the time evolution of the system is governed by the \emph{mass-action ODE system}
\begin{equation}
\label{eq:mas}
\frac{dx}{dt} = \sum_{j=1}^r k_j (\beta_j - \alpha_j) x^{\alpha_j} =: f(x),
\end{equation}
where $x^{\alpha_j} = \prod_{i=1}^n x_i^{\alpha_{ij}}$ is the monomial corresponding to the source complex of reaction $j$. The vector $\beta_j - \alpha_j$ is the \emph{reaction vector} of reaction $j$, and the \emph{stoichiometric subspace} $S = \mathrm{span}\{\beta_j - \alpha_j : j = 1, \ldots, r\} \subseteq \mathbb{R}^n$ is the subspace spanned by all reaction vectors.

Since each monomial $x^{\alpha_j}$ is non-negative for $x \in \mathbb{R}^n_{\geq 0}$, the positive orthant $\mathbb{R}^n_{\geq 0}$ is forward-invariant under \eqref{eq:mas}: concentrations that begin non-negative remain non-negative for all positive time.

\subsection{Stoichiometric Compatibility Classes}

Solutions of \eqref{eq:mas} are further confined by \emph{conservation laws}: linear functions of the concentrations that remain constant along any trajectory. Specifically, if $w \in \mathbb{R}^n$ is orthogonal to $S$ (i.e., $w \cdot (\beta_j - \alpha_j) = 0$ for all $j$), then $w \cdot x(t)$ is constant. The conservation laws therefore restrict each trajectory to an affine coset of the stoichiometric subspace.

\begin{definition}
For a given initial condition $x_0 \in \mathbb{R}^n_{\geq 0}$, the \emph{stoichiometric compatibility class} of $x_0$ is the set
\[
\mathcal{C} = (x_0 + S) \cap \mathbb{R}^n_{\geq 0}.
\]
\end{definition}

Each stoichiometric compatibility class $\mathcal{C}$ is a  convex polyhedron ; it is compact when the conservation laws bound all species concentrations, as they do for the futile-cycle models studied below. Its \emph{interior} $\mathrm{int}(\mathcal{C})$ consists of all points in $\mathcal{C}$ with strictly positive coordinates in the directions not fixed by the conservation laws, and its \emph{boundary} $\partial\mathcal{C} = \mathcal{C} \setminus \mathrm{int}(\mathcal{C})$ consists of points where at least one species concentration is zero.

\subsection{Steady States}

\begin{definition}\label{def:steadystates}
A point $x^* \in \mathbb{R}^n_{\geq 0}$ is a \emph{steady state} of \eqref{eq:mas} if $f(x^*) = 0$.
A steady state is \emph{positive} if $x^* \in \mathbb{R}^n_{> 0}$, and a \emph{boundary steady state} if $x^* \in \partial \mathbb{R}^n_{\geq 0}$, i.e., if at least one species concentration is zero. A steady state is \emph{dead} if all reaction rates vanish at $x^*$ (i.e., $k_j x^{*\alpha_j} = 0$ for all $j$), and \emph{living} otherwise.
\end{definition}

A steady state $x^*$ is \emph{locally asymptotically stable} (relative to a stoichiometric compatibility class $\mathcal{C}$) if every trajectory starting sufficiently close to $x^*$ in $\mathcal{C}$ converges to $x^*$ as $t\to\infty$. For interior steady states, local asymptotic stability can be assessed from the eigenvalues of the Jacobian $Df(x^*)$ restricted to $S$: $x^*$ is locally asymptotically stable if all such eigenvalues have strictly negative real part.

\subsection{Boundary Steady State Stability}

Assessing the stability of a boundary steady state $x^*$ requires special care, since the Jacobian of $f$ at $x^*$ may have eigenvalues with non-negative real part corresponding to directions that point outside $\mathcal{C}$. We use the \emph{next-generation matrix method}, adapted from mathematical epidemiology \cite{van2002reproduction, van2008further} to biochemical reaction networks \cite{avram2024advancing, avram2024stability, johnston2026boundary}.

The method works as follows. Partition the species so that $\tilde{x}$ are the \emph{vanishing species} (those with $x^*_i = 0$) and $\tilde{y}$ are the species with $x^*_i > 0$. The dynamics of the vanishing species near $x^*$ are governed by the linearization restricted to $\tilde{x}$. Write the right-hand side for the vanishing species as $\mathcal{F} - \mathcal{V}$, where $\mathcal{F}$  is chosen so that, together with $\mathcal{V}$, the resulting Jacobians satisfy the hypotheses of Theorem \ref{thm:ngm}. Let $F$ and $V$ denote the Jacobians of $\mathcal{F}$ and $\mathcal{V}$ with respect to $\tilde{x}$, evaluated at $x^*$. We recall that a matrix $M$ is a \emph{$Z$-matrix} if its off-diagonal entries are all non-positive.

\begin{theorem}[\cite{johnston2026boundary}]\label{thm:ngm}
Suppose $F$ is non-negative and nonzero, $V$ is an invertible $Z$-matrix, and $x^*$ is locally asymptotically stable with respect to the invariant subsystem at the intersection of the boundary face containing $x^*$ and the stoichiometric compatibility class $\mathcal{C}$. Then $x^*$ is locally asymptotically stable relative to $\mathcal{C}$ if $\rho(FV^{-1}) < 1$, and unstable if $\rho(FV^{-1}) > 1$, where $\rho(\cdot)$ denotes the spectral radius.
\end{theorem}

The quantity $\rho(FV^{-1})$ plays the role of the \emph{basic reproduction number} from epidemiology: when it is less than one the boundary steady state attracts nearby interior trajectories, and when it exceeds one the boundary steady state repels them. Note that, if the intersection of the boundary face containing $x^*$ and the stoichiometric compatibility class $\mathcal{C}$ consists of only the point $x^*$, then local stability holds trivially. For the networks studied in this paper, we examined the explicit Jacobian matrix and identified $\mathcal{F}$ and $\mathcal{V}$ that would satisfy the hypotheses of Theorem \ref{thm:ngm}. The general problem of establishing computationally efficient methods targeting complex networks remains open.

\subsection{Absolute Concentration Robustness}

\begin{definition}
A reaction network has \emph{absolute concentration robustness} (ACR) if for any choice of positive rate constants, the mass-action system has the following two properties:
\begin{enumerate}
    \item a positive steady state exists, and
    \item  at least one species coordinate is invariant across all positive steady states, i.e., $x_i^* = q$ for every positive steady state $x^*$.
\end{enumerate}
When these hold, $X_i$ is  an \emph{ACR species} and $q$ is  its \emph{ACR value}. A network may have ACR in more than one species coordinate.
\end{definition}

ACR is a strong form of robustness: the concentration of an ACR species at steady state is determined entirely by the rate constants, independent of initial conditions (and therefore of the total conserved quantities). The concept was introduced by Shinar and Feinberg \cite{shinar2010structural}, who also gave sufficient structural conditions for ACR in deficiency-one networks. Joshi and Nguyen \cite{joshi2026bifunctional} subsequently developed structural conditions that apply to enzymatic networks of arbitrary deficiency.

\section{Single-site futile cycle with a bifunctional enzyme}
\label{sec:SSP_1site}
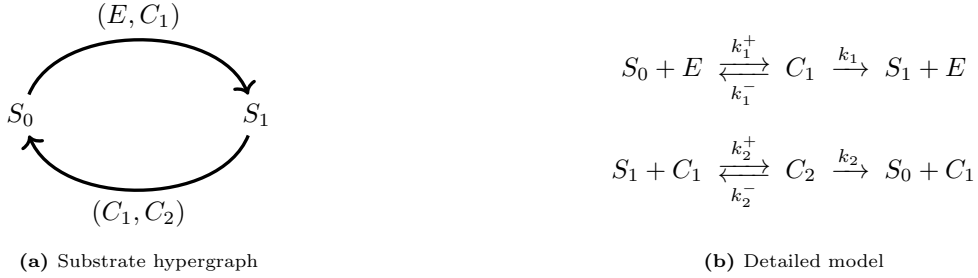
\begin{figure}[H]
    \centering
    \begin{subfigure}{0.40\textwidth}
        \centering
        \begin{tikzpicture}[baseline={(current bounding box.center)}]
            \node (S0)  {$S_0$};
            \node (S1) [right=2.5cm of S0] {$S_1$};
            \draw[->,line width=1.25] (S0) .. controls +(0.5,1.23) and +(-0.5,1.23) .. node[above] {$(E, C_1)$} (S1);
            \draw[->,line width=1.25] (S1) .. controls +(-0.5,-1.23) and +(0.5,-1.23) .. node[below] {$(C_1, C_2)$} (S0);
        \end{tikzpicture}
        \caption{Substrate hypergraph}
    \end{subfigure}%
    \hfill
    \begin{subfigure}{0.55\textwidth}
        \centering
        {
        \begin{gather*}
        S_0 + E \;\xrightleftarrows{k_1^+}{k_1^-}\; C_1 \;\xrightarrow{k_1}\; S_1 + E \\[6pt]
        S_1 + C_1 \;\xrightleftarrows{k_2^+}{k_2^-}\; C_2 \;\xrightarrow{k_2}\; S_0 + C_1
        \end{gather*}
        }
        \caption{Detailed model}
    \end{subfigure}
 \caption{Single-site futile cycle with bifunctional enzyme $C_1$ and its detailed model under the Henri-Michaelis-Menten mechanism.}
    \label{fig:futilecycle_singlesite_main}
\end{figure}
\begin{figure}[h!]
    \centering
    \includegraphics[width=0.8\linewidth]{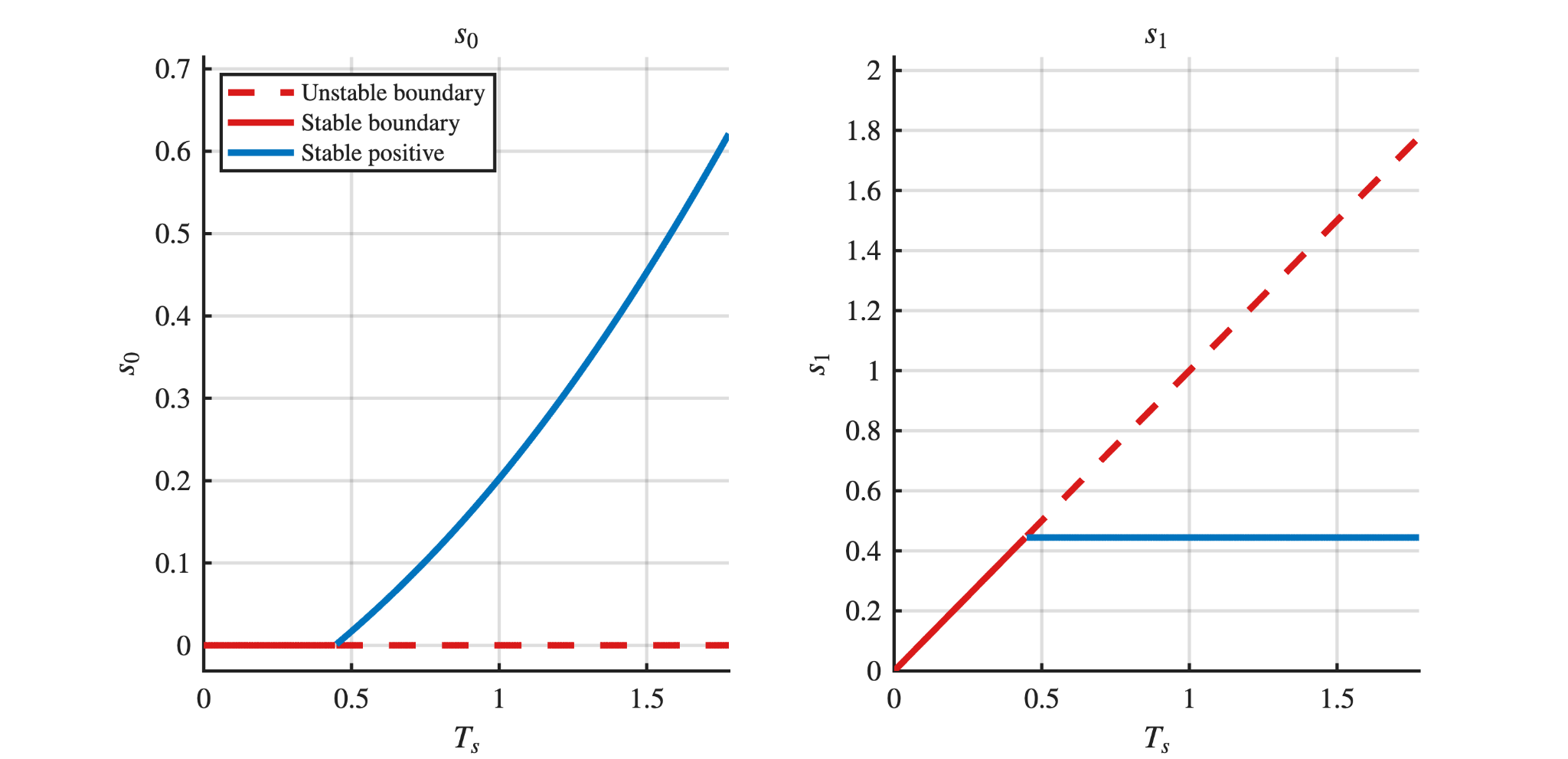}
    \caption{Bifurcation diagram for the single-site $(E,C_1)$ network showing $s_0$ and $s_1$ as functions of the total substrate $T_s$. The boundary steady state (red) is stable for $T_s < k_1/k_2^*$ and loses stability at a transcritical bifurcation where the unique stable positive steady state (blue) emerges. ACR in $S_1$ is reflected in the flat positive steady state branch in the $s_1$ panel. Parameter values: $k_1^+ = 2$, $k_1^- = 1$, $k_1 = 1$, $k_2^+ = 3$, $k_2^- = 0.5$, $k_2 = 1.5$, $T_e = 1$.}
    \label{fig:ec1}
\end{figure}

A single-site futile cycle is one where a substrate is reversibly converted from one form to another, for instance from the unphosphorylated form $S_0$ to the phosphorylated form $S_1$. In a bifunctional enzyme network, a single enzyme species catalyzes both the forward and reverse transformations. Here, the common enzyme $E$ catalyzes the forward reaction $S_0 \to S_1$, forming the intermediate compound $C_1 = S_0\text{-}E$. The compound $C_1$ is bifunctional: it simultaneously serves as the intermediate of the forward reaction and as the enzyme for the reverse reaction $S_1 \to S_0$. We denote this network by $(E, C_1)$, where the two entries list the enzymes for the forward and reverse reactions, respectively. The IDHKP-IDH signaling system in \textit{Escherichia coli} is a prominent biochemical example of this architecture \cite{shinar2009robustness}. We refer the reader to Figure \ref{fig:futilecycle_singlesite_main} for the substrate hypergraph and the detailed reaction network under the Henri-Michaelis-Menten mechanism (the specific reaction steps in which enzyme and substrate bind to form an intermediate compound, which then releases the product). This is one particular \emph{detailed model} of the composite reactions; other detailed models are possible, but we assume the Henri-Michaelis-Menten mechanism throughout this paper.

The full mass-action system and its derivation are given in the Supplementary Material. 
We introduce the shorthand notation
\begin{equation} \label{not:k*}
    k_\ell^* \coloneqq \frac{k_\ell k_\ell^+}{k_\ell + k_\ell^-},
\end{equation}
which represents the effective catalytic rate of reaction $\ell$.

The main results for the single-site network are as follows.
\begin{theorem}\label{prop:BSS_1site}
    For any choice of positive rate constants and $T_e, T_s>0$, there is a unique boundary steady state with $e=T_e$ and $s_1=T_s$, which is locally stable if and only if $T_s < k_1/k_2^*$.
\end{theorem}
\begin{theorem}\label{thm:singlesite}
    For any choice of positive rate constants and every $T_e > 0$, a positive steady state exists if and only if $T_s > k_1/k_2^*$. Moreover, this positive steady state is unique and locally asymptotically stable.
\end{theorem}
 
We give the boundary-stability proof of Theorem~\ref{prop:BSS_1site} here as a representative application of the next-generation matrix method; the Supplementary Material contains the same calculation together with the proof of Theorem~\ref{thm:singlesite} and the analogous dual-site calculations.

\begin{proof}[Proof of Theorem~\ref{prop:BSS_1site}]
At a boundary steady state, the steady-state relations imply $c_1=c_2=0$. The conservation laws then give $e=T_e$, and since $s_0e=0$ with $T_e>0$, we obtain $s_0=0$ and $s_1=T_s$. Thus the boundary steady state is unique.

The vanishing species are $s_0,c_1,c_2$. Restricting their dynamics and writing the right-hand side as $\mathcal F-\mathcal V$, choose
\[
\mathcal F=\langle k_2c_2,0,0\rangle,
\]
\[
\mathcal V=\langle k_1^+s_0e-k_1^-c_1,\;-k_1^+s_0e+(k_1^-+k_1)c_1+k_2^+s_1c_1-(k_2^-+k_2)c_2,\;-k_2^+s_1c_1+(k_2^-+k_2)c_2\rangle.
\]
Taking Jacobians with respect to $(s_0,c_1,c_2)$ and evaluating at the boundary steady state gives
\[
F=\begin{bmatrix}0&0&k_2\\0&0&0\\0&0&0\end{bmatrix},\qquad
V=\begin{bmatrix}
k_1^+T_e&-k_1^-&0\\
-k_1^+T_e&k_2^+T_s+k_1^-+k_1&-(k_2^-+k_2)\\
0&-k_2^+T_s&k_2^-+k_2
\end{bmatrix}.
\]
A direct inversion yields
\[
\rho(FV^{-1})=\frac{k_2^*}{k_1}T_s.
\]
The boundary face intersects each stoichiometric compatibility class only at this steady state, so the within-face stability hypothesis of Theorem~\ref{thm:ngm} is automatic. Hence the boundary steady state is locally asymptotically stable when $\rho(FV^{-1})<1$, equivalently $T_s<k_1/k_2^*$, and unstable when $T_s>k_1/k_2^*$.
\end{proof}

Together, these results imply a transcritical bifurcation at $T_s = k_1/k_2^*$, illustrated in Figure \ref{fig:ec1}.

\begin{theorem}[ACR for the $(E,C_1)$ network \cite{shinar2010structural, joshi2026bifunctional}]\label{thm:ACR_1site}
The single-site network $(E,C_1)$ has ACR in species $S_1$ with ACR value $k_1/k_2^*$.
\end{theorem}

This result was first established for the $(E,C_1)$ network in \cite{shinar2010structural} using the structural deficiency-based theory of ACR. The hypergraph approach of \cite{joshi2026bifunctional} provides an alternative, more direct derivation from the network structure.

The single-site futile cycle is a subnetwork of the dual-site futile cycle, so these results will be useful when studying certain boundary steady states of the dual-site system. 

\section{Dual-site futile cycle with bifunctional enzymes}
\label{sec:dualsite}

Multisite phosphorylation is a common biochemical motif, and the dual-site futile cycle is the natural extension of the single-site system in the previous section. 
The range of possible dynamics is considerably richer when going from one site to two sites, as the remainder of the paper will reveal. 
We consider the situation where the forward steps are catalyzed by a single common enzyme $E$ while the backward steps are catalyzed by one of the intermediate compounds produced in the forward steps. 
In other words, the backward step enzymes $\{\EE_1, \EE_2\}$ are drawn from the set $\{C_1, C_2\}$. 
Fixing the order of reactions $(S_0 \to S_1, S_1 \to S_2, S_2 \to S_1, S_1 \to S_0)$, each network is specified by simply listing the enzyme that acts on the corresponding reaction. Since each of $\EE_1$ and $\EE_2$ can be one of two possible compounds, we have 4 distinct networks, namely, $(E,E,C_1,C_1)$, $(E,E,C_1,C_2)$, $(E,E,C_2,C_1)$, and $(E,E,C_2,C_2)$.
See Figure \ref{fig:dualsite} for the substrate hypergraph and detailed reaction network under the Henri-Michaelis-Menten mechanism, which we assume throughout.

Our analysis follows a unified approach for all four networks. We derive a steady state parameterization in terms of a single variable $u \coloneqq c_1/c_2$, and show that the positive steady states are in one-to-one correspondence with the positive roots of a single-variable polynomial $P(u)$ whose coefficients depend on the rate constants and total concentrations $T_e$, $T_s$ (Theorem \ref{thm:1-1}). We then study the number and existence of positive roots using Descartes' rule of signs, treating $T_s$ as a bifurcation parameter while fixing the rate constants and $T_e$. Stability of each steady state is determined by computing the eigenvalues of the Jacobian matrix.

 Before establishing steady state parameterizations for the four networks, we first  recall their ACR properties, which  were established previously in \cite{joshi2024bifunctional, joshi2026bifunctional}. These ACR results are not new contributions of the present paper; we include them here because they provide a structural axis for comparing the four architectures with the new steady-state multiplicity, stability, and bifurcation results developed below.

\begin{theorem}[ACR for dual-site networks \cite{joshi2024bifunctional, joshi2026bifunctional}]\label{thm:ACR}
The following ACR results hold for the four dual-site networks.
\begin{enumerate}
    \item The network $(E,E,C_2,C_1)$ has ACR in species $S_1$ and $S_2$ with ACR values $k_1/k_4^*$ and $k_2/k_3^*$ respectively.
    \item The network $(E,E,C_1,C_1)$ has ACR in species $S_1$ with ACR value $k_1/k_4^*$.
    \item The network $(E,E,C_1,C_2)$ does not have ACR in any species.
    \item The network $(E,E,C_2,C_2)$ has ACR in species $S_2$ with ACR value $k_2/k_3^*$.
\end{enumerate}
\end{theorem}

The ACR properties for the $(E,E,C_1,C_1)$ and $(E,E,C_2,C_2)$ networks were first identified in \cite{joshi2024bifunctional}. All four results follow from the structure of the substrate hypergraph via \cite{joshi2026bifunctional}, and hold for any detailed model of the respective hypergraphs, not only for the Henri-Michaelis-Menten mechanism studied here.

\subsection{Steady state parameterization}\label{sec:SSP_2site}
For the reverse reactions, let $\varepsilon_1$ and $\varepsilon_2$ denote the concentrations of the bifunctional enzymes $\EE_1$ and $\EE_2$, respectively. Thus $\varepsilon_i$ is either $c_1$ or $c_2$. We write $\bm{1}_{C_j}(\EE_i)=1$ when $\EE_i=C_j$ and $0$ otherwise. This indicator notation lets one mass-action system represent all four choices of bifunctional enzymes.
The mass-action system of the detailed model in Figure \ref{fig:dualsite} can be written as follows:
\begin{equation}\label{eq:ODE_2site_GL}
\begin{aligned}
&\frac{de}{dt}    = -G_1-G_2, \\
&\frac{dc_1}{dt} = G_1 - \bm{1}_{C_1}(\EE_1)G_3- \bm{1}_{C_1}(\EE_2)G_4,\quad \frac{dc_2}{dt} = G_2 - \bm{1}_{C_2}(\EE_1)G_3- \bm{1}_{C_2}(\EE_2)G_4\\
&\frac{dc_3}{dt}  = G_3,\quad \frac{dc_4}{dt}=G_4,\\
&\frac{ds_0}{dt} = L_0- G_1,\quad \frac{ds_1}{dt}=L_1-G_2-G_4,\quad \frac{ds_2}{dt}=L_2-G_4.
\end{aligned}
\end{equation}
where
\begin{equation}\label{eq:GL}
\begin{aligned}
& G_{\ell} = k_{\ell}^+ s_{T_\ell}\varepsilon_\ell - (k_\ell^- +k_\ell)c_\ell, \quad &&\text{for} \quad \ell \in \{1,2,3,4\} \\
& L_i =  \sum_{\ell: i= H_\ell} k_{\ell} c_{\ell} - \sum_{\ell: i=T_\ell} k_{\ell}c_{\ell} , \quad &&\text{for} \quad i\in\{0,1,2\},\\
\end{aligned}
\end{equation}
with $H_\ell,T_\ell$ denoting the indices of the substrates at the  product and reactant ends of reaction step $\ell$, respectively. 

The conservation laws of the detailed model in Figure \ref{fig:dualsite} are:
\begin{equation} \label{eq:cons_nof}
    \begin{aligned}
        &T_s = s_0 + s_1 + s_2 + c_1 + c_2 + 2c_3 + 2 c_4, \\
        &T_e = e + c_1 + c_2 + c_3 + c_4. 
    \end{aligned}
\end{equation}
We have the following lemma based on the results in \cite{joshi2026bifunctional}.

\begin{lemma}[\cite{joshi2026bifunctional}, Lemma S12.3]\label{lem:GL_2site}
    The steady state equations for \eqref{eq:ODE_2site_GL} are equivalent to
\begin{align*}
    &\text{} \text{(reaction-step balances)} \quad G_\ell=0 \quad \text{for}\quad \ell\in\{1,2,3,4\}\\
    &\text{} \text{(substrate balances)}\quad L_i=0 \quad \text{for} \quad i\in\{0,1,2\}.
\end{align*}
\end{lemma}
The following lemma follows directly from Lemma \ref{lem:GL_2site} and the notation \ref{not:k*}.
\begin{lemma}\label{lem:SSP_2site}
    All steady states of \eqref{eq:ODE_2site_GL} must satisfy:
    \[
k_1^*s_0 e = k_1c_1 = k_4^*s_1\varepsilon_1 = k_4c_4, \quad \quad k_2^*s_1 e= k_2c_2 = k_3^*s_2\varepsilon_2 = k_3c_3.  
    \]
\end{lemma}

We wish to find a parameterization of the positive steady states. 
We will initiate the process with the choice of parameters $e, c_1, c_2$. The remaining coordinates at steady state are then given by: 
\begin{equation}
    \begin{aligned}
        s_0 = \frac{k_1}{k_1^*} \frac{c_1}{e},  
        \quad 
        s_1 = \frac{k_2}{k_2^*} \frac{c_2}{e} = \frac{k_1}{k_4^*} \frac{c_1}{\varepsilon_1}, 
        \quad 
        s_2 =  \frac{k_2}{k_3^*} \frac{c_2}{\varepsilon_2}, 
        \quad 
        c_4 = \frac{k_1}{k_4} c_1, 
        \quad 
        c_3 = \frac{k_2}{k_3} c_2. 
    \end{aligned}
\end{equation}
Note that $s_1$ can be parameterized in two different ways using either $c_1$ or $c_2$ since it appears in both chains of equality. 
We may use this to eliminate $e$ as a primary parameter. Our goal is to express every positive steady-state coordinate in terms of a single ratio $u$, and then use the two conservation laws to reduce the existence problem to one polynomial equation in $u$.
Define the following convenient change of variables: 
\begin{equation}\label{eq:change_vars}
    c \coloneqq c_2, \quad u \coloneqq \frac{c_1}{c_2}, \quad v_1 \coloneqq \frac{c_1}{\varepsilon_1} \in \{1,u\}, \quad v_2 \coloneqq \frac{c_2}{\varepsilon_2} \in \left\{1,\frac{1}{u}\right\}. 
\end{equation}
Indeed, if $\varepsilon_1=c_1$ then $v_1=1$, whereas if $\varepsilon_1=c_2$ then $v_1=c_1/c_2=u$. Similarly, $v_2=1$ when $\varepsilon_2=c_2$ and $v_2=c_2/c_1=1/u$ when $\varepsilon_2=c_1$. Thus the braces in \eqref{eq:change_vars} list the two possible values induced by the choice of bifunctional enzyme, rather than introducing additional free parameters.
The positive steady states of \eqref{eq:ODE_2site_GL} can be parameterized in terms of $\{c,u\}$ as follows. 
\begin{equation} \label{eq:parameters_cu}
    \begin{aligned}
        c_2 &= c,
        \quad
        c_1 = cu, 
        \quad
        e = \gamma \frac{c}{v_1}, 
        \quad 
        s_0 = \alpha_0 u v_1,  
        \quad 
        s_1 = \alpha_1 v_1, 
        \quad 
        s_2 =  \alpha_2 v_2, 
        \quad 
        c_4 = \beta_4 uc, 
        \quad 
        c_3 = \beta_3 c
    \end{aligned}
\end{equation}
where
\[
\gamma=\frac{k_2}{k_1} \frac{k_4^*}{k_2^*},\quad\alpha_0=\frac{k_1}{k_1^*} \frac{k_1}{k_2} \frac{k_2^*}{k_4^*},\quad \alpha_1=\frac{k_1}{k_4^*},\quad \alpha_2=\frac{k_2}{k_3^*}.\quad \beta_4=\frac{k_1}{k_4}, \quad \beta_3=\frac{k_2}{k_3}.
\]
Plugging the steady state parameterization \eqref{eq:parameters_cu} into the conservation laws \eqref{eq:cons_nof} yields: 
\begin{equation} \label{eq:cons_nof_par}
    \begin{aligned}
        &T_s = \alpha_0 uv_1 + \alpha_1 v_1 + \alpha_2 v_2 + c\left((1+ 2\beta_4)u+(1+ 2\beta_3)\right)\\
        &T_e = c \left( \frac{\gamma}{v_1} + (1+\beta_4)u + (1+\beta_3)\right). 
    \end{aligned}
\end{equation}
We combine the two equations in \eqref{eq:cons_nof_par} after eliminating $c$, which gives a single equation in one variable $u$:  
\begin{equation}\label{eq:Ts_in_u}
\begin{aligned}
        T_s &= \alpha_0 uv_1 + \alpha_1 v_1 + \alpha_2 v_2 + T_e \frac{(1+2\beta_4)u+1+2\beta_3}{ \frac{\gamma}{v_1} + (1+\beta_4)u + 1+\beta_3} \\
        &= \alpha_0 uv_1 + \alpha_1 v_1 + \frac{1}{u}\alpha_2 (uv_2) + T_e v_1  \left(\frac{(1+2\beta_4)u+1+2\beta_3}{ \gamma + v_1\left((1+\beta_4)u + 1+\beta_3\right)}\right) \\
\end{aligned}
\end{equation}
For convenience, we define the following variable
\begin{equation}
    w_2:=uv_2 \in \{1,u\}
\end{equation}
Thus $w_2$ is simply a bookkeeping variable for the factor $uv_2$ after denominators are cleared: $w_2=1$ when $\varepsilon_2=c_1$ and $w_2=u$ when $\varepsilon_2=c_2$.
and the parameters 
\begin{equation}
    a_i=1+\beta_i,\quad  b_i=1+2\beta_i, \quad \text{for}\quad i\in\{3,4\}. 
\end{equation}
We multiply both sides of \eqref{eq:Ts_in_u} by $u \left( \gamma + v_1\left((1+\beta_4)u + 1+\beta_3\right)\right)$ and rearrange to get a polynomial equation in $u$. We denote it by $P_{\EE_2,\EE_1}(u)$; the subscripts record the choices of bifunctional enzymes and therefore determine the cases for $v_1$ and $v_2$: 
\begin{equation} \label{eq:mainpoly_u}
\begin{aligned}
        P_{\EE_2,\EE_1}(u):=(\alpha_0 u^2v_1 + \alpha_1 u v_1 + \alpha_2 w_2 - T_s u)\left( \gamma + v_1\left(a_4u + a_3\right)\right) + (b_4u+b_3)T_e u v_1 = 0. 
\end{aligned}
\end{equation}

\begin{theorem}\label{thm:1-1}
 For fixed rate constants and fixed $T_e > 0$ and $T_s > 0$, 
    there is a one-to-one correspondence between the positive roots of  $P_{\EE_2,\EE_1}(u)$ and the positive steady states of \eqref{eq:ODE_2site_GL}. 
\end{theorem}
\begin{proof}
    From \eqref{eq:cons_nof_par}, we have
    \[
    c=\frac{T_e}{\frac{\gamma}{v_1}+(1+\beta_4)u+(1+\beta_3)}.
    \]
    Since $v_1\in\{1,u\}$, for each solution $u>0$ of \eqref{eq:mainpoly_u}, there is exactly one solution $c>0$ of \eqref{eq:cons_nof_par}. From the parameterization \eqref{eq:parameters_cu}, there is exactly one corresponding positive steady state of \eqref{eq:ODE_2site_GL}. Conversely, every positive steady state defines $u=c_1/c_2>0$ and satisfies the two conservation laws, hence \eqref{eq:mainpoly_u}. Therefore the correspondence is bijective.
\end{proof}

For the examples used below, to show that the upper bounds are attained, we searched over positive parameter values guided by the coefficient sign and discriminant conditions derived for the corresponding polynomial, and then substituted the resulting values back into the polynomial to verify the required number of positive roots. Stability of the resulting steady states was checked from the eigenvalues of the reduced Jacobian, as described explicitly where multistability is claimed.

Next, we study the existence, uniqueness, and stability of the steady states of the four networks in the dual-site class \eqref{eq:ODE_2site_GL}.

\subsection{The $(E,E,C_2,C_1)$ network}

\begin{figure}[h!]
    \centering
\includegraphics[width=1\linewidth]{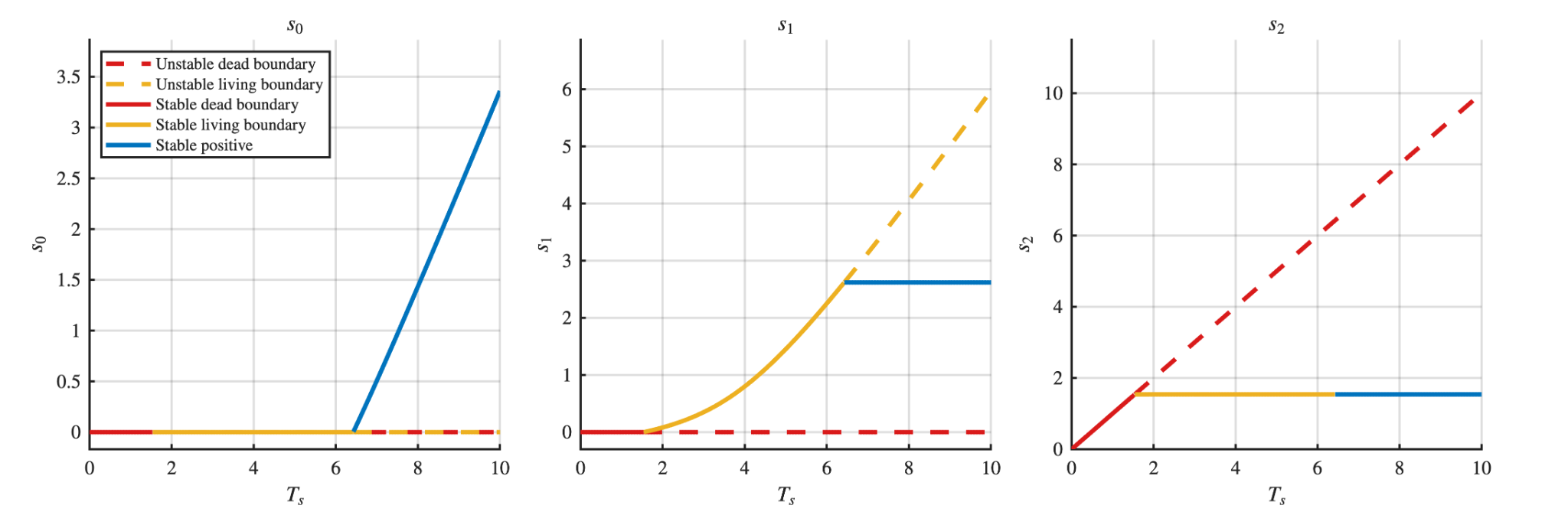}
    \caption{Bifurcation diagram for the $(E,E,C_2,C_1)$ network showing $s_0$, $s_1$, $s_2$ in the case of one positive steady state. The dead boundary steady state (red, stable for $T_s < k_2/k_3^*$, unstable for $T_s > k_2/k_3^*$), the living boundary steady state (orange, emerging at $T_s = k_2/k_3^*$ via a transcritical bifurcation, stable for $k_2/k_3^* < T_s < T_s^{2,1}$, unstable for $T_s > T_s^{2,1}$), and the positive steady state (blue, emerging at $T_s = T_s^{2,1}$ via a transcritical bifurcation with the living boundary steady state) are shown. ACR in both $S_1$ and $S_2$ is visible as flat branches in the corresponding panels. Parameter values: $k_1^+ = 3.33$, $k_1^- = 1.04$, $k_1 = 1.04$, $k_2^+ = 2$, $k_2^- = 1$, $k_2 = 1$, $k_3^+ = 1.3$, $k_3^- = 1$, $k_3 = 1$, $k_4^+ = 1.08$, $k_4^- = 2.15$, $k_4 = 1.25$, $T_e = 1.8$.}
    \label{fig:eec2c1_1pss}
\end{figure}

\begin{figure}[h!]
    \centering
\includegraphics[width=1\linewidth]{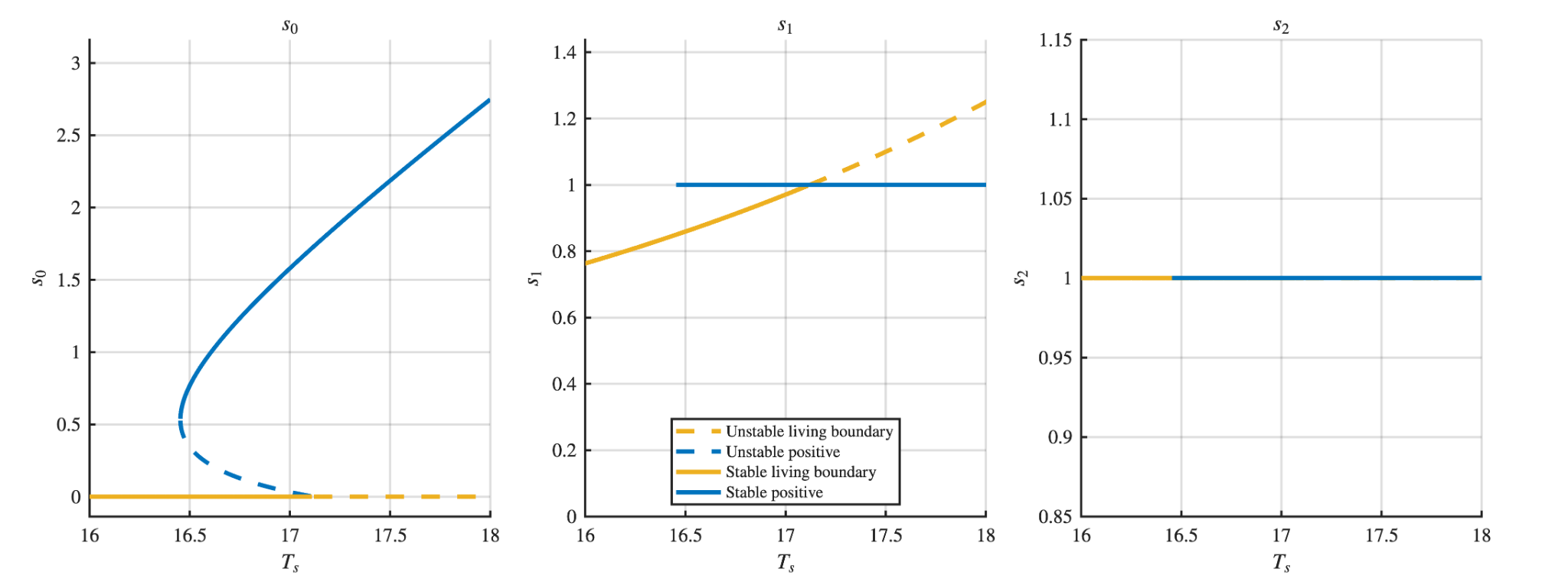}
    \caption{Bifurcation diagram for the $(E,E,C_2,C_1)$ network showing $s_0$, $s_1$, $s_2$ in the case of two positive steady states. The dead boundary steady state branch is omitted (see Figure~\ref{fig:eec2c1_1pss}). The living boundary steady state (orange) and two positive steady state branches (blue) are shown: a stable and an unstable positive steady state emerge together at a saddle-node bifurcation, and the unstable branch collides with the living boundary steady state at $T_s = T_s^{2,1}$ in a backward transcritical bifurcation. ACR in both $S_1$ and $S_2$ is visible as flat branches. Parameter values: $k_1^+ = 38.33$, $k_1^- = 2.04$, $k_1 = 2.04$, $k_2^+ = 6$, $k_2^- = 3$, $k_2 = 3$, $k_3^+ = 6$, $k_3^- = 1$, $k_3 = 1$, $k_4^+ = 4.08$, $k_4^- = 8.15$, $k_4 = 8.15$, $T_e = 10.8$.}
    \label{fig:eec2c1_2pss}
\end{figure}

In this subsection, we study the dynamics of the $(E,E,C_2,C_1)$ network. Since the enzymes of the reverse chain are $\EE_2=C_2$ and $\EE_1=C_1$, we have $v_1=1$ and $v_2=1$ (thus $w_2=u$). At positive steady states, equation \eqref{eq:mainpoly_u} becomes $Q(u)=0$ where
    \begin{equation} \label{eq:poly_c1c2}
Q(u):=\frac{P_{C_2,C_1}(u)}{u}=(u^{2}\,\alpha_{0} a_4
+ u\!\left(
b_4T_{e}
+ (a_3+\gamma) \alpha_{0}
+ a_4(\alpha_{1}
+ \alpha_{2}
- T_s)
\right)
+ \left(
b_{3}T_{e}
+ (a_3+\gamma)(\alpha_{1} + \alpha_2-T_s)
\right).
    \end{equation}

\begin{theorem}\label{thm:EEC2C1}
    Consider the model $(E,E,C_2,C_1)$. 
    \begin{enumerate}
        \item (Dead boundary steady state) For any choice of positive rate constants and $T_e,T_s>0$, there is a unique dead boundary steady state with $e=T_e$ and $s_2=T_s$. This steady state is locally stable if and only if $T_s<k_2/k_3^*$.
        \item (Living boundary steady state) There exists a unique living boundary steady state if $T_s>k_2/k_3^*$. This steady state is locally stable if and only if $T_s < T_s^{2,1}$ where
        \[
        T_s^{2,1}=\alpha_1+\alpha_2+\frac{b_3T_e}{a_3+\gamma}=\frac{k_1}{k_4^*} + \frac{k_2}{k_3^*} + \frac{\left( 1 + \frac{2k_2}{k_3}\right)  T_e}{ \left(1 + \frac{k_2}{k_3} \right)+\frac{k_2}{k_1} \frac{k_4^*}{k_2^*}}.
        \]
        The living boundary steady state corresponds to the sub-network $S_1 \rightleftarrows S_2$ being ON while the reactions producing and consuming $S_0$ are switched OFF.
        \item (Positive steady state existence) For any choice of positive rate constants and any $T_e>0$, there is exactly one positive steady state for all $T_s > T_s^{2,1}$.
        \item (Positive steady state capacity) For any choice of positive rate constants and $T_e, T_s>0$, there are at most two positive steady states. 
        Furthermore, there exists a choice of positive rate constants and $T_e, T_s>0$ for which there are two positive steady states.
        
    \end{enumerate}
\end{theorem}

\begin{proof}
The proofs of part 1 and 2 follow from the next-generation matrix method. The detailed calculations can be found in the Supplementary Material.

For part 3, we observe that when $T_s>T_s^{2,1}$, we have
\[
b_{3}T_{e}
+ (a_3+\gamma)(\alpha_{1} + \alpha_2-T_s)<0.
\]
Thus there is exactly one sign change in $Q(u)$, which  implies by Descartes' rule of signs that $Q(u)$ has at most one positive root. Its constant term is negative while its leading coefficient is positive, so $Q(0)<0$ and $Q(u)>0$ for sufficiently large $u$; continuity therefore guarantees at least one positive root. Hence $Q(u)$ has exactly one positive root. From Theorem \ref{thm:1-1}, there is exactly one positive steady state.

For part 4, since $Q(u)$ is a quadratic polynomial, there are at most two positive roots. Thus by Theorem \ref{thm:1-1}, there are at most two positive steady states. The necessary and sufficient condition for the existence of two positive steady states is
\begin{equation}\label{eq:iff_2pss_c1c2}
\text{Disc}(Q) >0 \quad \text{and} \quad b_{3}T_{e}
+ (a_3+\gamma)(\alpha_{1} + \alpha_2-T_s)>0,
\end{equation}
where $\text{Disc}(Q)$ is the discriminant of the polynomial $Q(u)$. A particular choice of positive rate constants and $T_s,T_e$ that satisfy \eqref{eq:iff_2pss_c1c2} is given as follows: $k_1^+ = 38.33$, $k_1^- = 2.04$, $k_1 = 2.04$, $k_2^+ = 6$, $k_2^- = 3$, $k_2 = 3$, $k_3^+ = 6$, $k_3^- = 1$, $k_3 = 1$, $k_4^+ = 4.08$, $k_4^- = 8.15$, $k_4 = 8.15$, $T_e = 10.8$, and $T_s=17$.

\end{proof}

\begin{remark}[Sufficient condition that preclude multistationarity]
We observe that the condition for $Q(u)$ to have two sign changes is 
\begin{equation}\label{eq:c1c2_twoPSS}
    \frac{b_4T_e + \alpha_0(a_3+ \gamma)}{a_4} < T_s - \alpha_1 - \alpha_2 < \frac{b_3T_e}{a_3 + \gamma}.
\end{equation}
There exists a choice of positive rate constants and total concentrations satisfying \eqref{eq:c1c2_twoPSS} if and only if 
\[
\frac{b_4T_e + \alpha_0(a_3+ \gamma)}{a_4}<\frac{b_3T_e}{a_3 + \gamma} \iff \alpha_0(a_3+\gamma) < T_e\frac{a_4b_3-b_4(a_3+\gamma)}{a_3+\gamma}.
\]
Thus, it is not possible for $Q(u)$ to have two positive roots (and thus there cannot be two positive steady states) if
\[
a_4b_3 < b_4(a_3+\gamma) \quad \text{or} \quad T_e<\frac{\alpha_0(a_3+\gamma)}{\frac{a_4b_3}{a_3+\gamma}-b_4}.
\]
\end{remark}

In Figures \ref{fig:eec2c1_1pss} and \ref{fig:eec2c1_2pss}, we present the bifurcation diagrams (as $T_s$ varies) of the $(E,E,C_2,C_1)$ network with two different choices of rate constants and $T_e$. 

Figure \ref{fig:eec2c1_1pss} illustrates the dynamics of the $(E,E,C_2,C_1)$ network when there is at most one positive steady state. In particular, there is a transcritical bifurcation at $T_s=k_2/k_3^*$ where the dead boundary steady state becomes unstable and the living boundary steady state emerges and is stable. There is another transcritical bifurcation at $T_s=T_s^{2,1}$ where the living boundary steady state becomes unstable and the positive steady state emerges and is stable.

Figure \ref{fig:eec2c1_2pss} illustrates the dynamics of the $(E,E,C_2,C_1)$ network when there are two positive steady states for some choices of $T_s$. There is still a transcritical bifurcation at $T_s=k_2/k_3^*$ where the living boundary steady state emerges. We omit this from Figure \ref{fig:eec2c1_2pss} since this behavior is observed in Figure \ref{fig:eec2c1_1pss} already. Next, there is a saddle-node bifurcation at $\text{Disc}(Q)=0$ where a stable and an unstable positive steady states emerge. Finally, there is a transcritical bifurcation at $T_s=T_s^{2,1}$ where the unstable positive steady state meets the living boundary steady state. This behavior is referred to as \textit{backward bifurcation} (for example, see \cite{gumel2012causes}) and can cause \textit{hysteresis} (for example, see \cite{strogatz2024nonlinear}) when parameters are varied.

\subsection{The $(E,E,C_1,C_1)$ network}

\begin{figure}[h!]
    \centering
\includegraphics[width=1\linewidth]{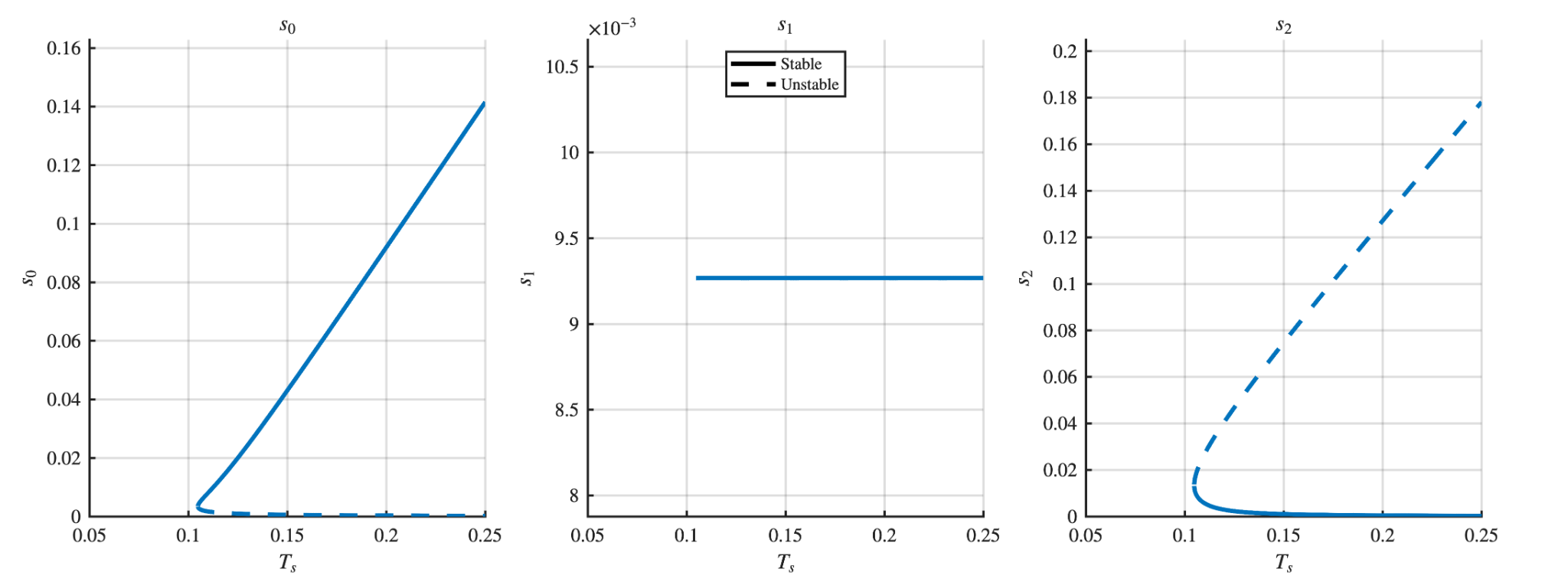}
    \caption{Bifurcation diagram for the $(E,E,C_1,C_1)$ network showing $s_0$, $s_1$, $s_2$ and multistationarity with ACR in species $S_1$. The dead boundary steady state, which is stable for all values of $T_s$, is omitted. A stable and an unstable positive steady state (blue) are born simultaneously in a saddle-node bifurcation. 
    ACR in $S_1$ is visible as a flat branch in the $s_1$ panel. Parameter values: $k_1^+ = 1370$, $k_1^- = 2.7$, $k_1 = 2.7$, $k_2^+ = 214$, $k_2^- = 1$, $k_2 = 1$, $k_3^+ = 168$, $k_3^- = 1$, $k_3 = 1$, $k_4^+ = 582.7$, $k_4^- = 1.36$, $k_4 = 1.36$, $T_e = 0.06$.}
    \label{fig:eec1c1}
\end{figure}

In this subsection, we study the dynamics of the $(E,E,C_1,C_1)$ network. Since the enzymes of the reverse chain are $\EE_2=\EE_1=C_1$, we have $v_1=1$ and $v_2=1/u$ (thus $w_2=1$). At positive steady states, equation \eqref{eq:mainpoly_u} becomes $P_{C_1,C_1}(u)=0$ where
\begin{equation*}
P_{C_1,C_1}(u)=u^{3}a_4\alpha_{0}
+ u^{2}\!\left(
b_{4}T_{e}
+ (a_3+\gamma) \alpha_{0}
- a_4(T_s - \alpha_{1})
\right)
+ u\!\left(
b_{3}T_{e}
+ a_4\alpha_{2}
- (a_3+\gamma) (T_{s}-\alpha_1)
\right)
+ (a_3+\gamma) \alpha_{2}.
\end{equation*}

\begin{theorem}\label{thm:EEC1C1}
    Consider the model $(E,E,C_1,C_1)$. 
    \begin{enumerate}
        \item (Boundary steady state) For any choice of positive rate constants and $T_e,T_s>0$, there is a unique boundary steady state with $e=T_e$ and $s_2=T_s$. This boundary steady state is locally stable for any choice of positive rate constants and $T_e,T_s>0$.
        \item (Positive steady state existence) For any choice of positive rate constants and any $T_e>0$, a positive steady state exists for any $T_s > T_s^{1,1}$, where
        \[
        T_s^{1,1} = \alpha_0 + \alpha_1 + \alpha_2 + \frac{T_e(b_4+b_3)}{a_4+a_3+\gamma}.
        \]
        Furthermore, the number of positive steady states (including multiplicity) is even.
        \item (Positive steady state capacity) For any choice of positive rate constants and $T_e, T_s>0$, there are at most two positive steady states. 
        
    \end{enumerate}
\end{theorem}

\begin{proof}
The proofs of part 1 follows from the next-generation matrix method. The detailed calculations can be found in the Supplementary Material. Part 3 follows from the fact that $P_{C_1,C_1}(u)$ has at most two sign changes.

For part 2, we observe that $P_{C_1,C_1}(0)=(a_3+\gamma)\alpha_2>0$ and $P_{C_1,C_1}(u)\sim u^3a_4\alpha_0 >0$ for sufficiently large $u$. Thus the number of positive roots of $P_{C_1,C_1}(u)$ (and the number of positive steady states) must be even. To show that a positive root exists for large enough $T_s$, we consider
\begin{align*}
P_{C_1,C_1}(1)&=a_4\alpha_{0}
+ 
b_{4}T_{e}
+ (a_3+\gamma) \alpha_{0}
- a_4(T_s - \alpha_{1})
+ 
b_{3}T_{e}
+ a_4\alpha_{2}
- (a_3+\gamma) (T_{s}-\alpha_1)
+ (a_3+\gamma) \alpha_{2}\\
&=(a_4+a_3+\gamma)(\alpha_0+\alpha_2)+T_e(b_4+b_3)-(a_4+a_3+\gamma)(T_s-\alpha_1).
\end{align*}
Thus if we let
\[
T_s^{1,1} = \alpha_0+\alpha_1+\alpha_2+\frac{T_e(b_4+b_3)}{a_4+a_3+\gamma}
\]
then $P_{C_1,C_1}(1)<0$ for any $T_s>T_s^{1,1}$. As a result, $P_{C_1,C_1}(u)$ has a positive root by the intermediate value theorem, and thus a positive steady state exists by Theorem \ref{thm:1-1}.
\end{proof}

Figure \ref{fig:eec1c1} illustrates the dynamics of the $(E,E,C_1,C_1)$ network for a specific choice of rate constants and total enzyme $T_e$. The boundary steady state, which is stable for any value of $T_s$, is omitted from the figure. We observe a saddle-node bifurcation at which a stable and an unstable positive steady state are born simultaneously, consistent with Theorem \ref{thm:EEC1C1}. As $T_s$ increases, the stable positive steady state shifts toward higher concentrations of $u$, $c_1$, $c_4$, and $s_0$, and lower concentrations of $c_2$, $c_3$, $e$, and $s_2$, reflecting the progressive shift of substrate toward the $S_0$ form. Notably, $s_2$ decreases as $T_s$ increases — a negative dose-response in the final product — while $s_1$ remains fixed at its ACR value $k_1/k_4^*$.

\begin{remark}\label{rem:eec1c1_asymptotics}
The $(E,E,C_1,C_1)$ network exhibits a particularly striking asymptotic behavior. From the parameterization \eqref{eq:parameters_cu} with $v_1=1$ and $v_2=1/u$, the positive steady state coordinates are $s_0 = \alpha_0 u$, $s_1 = \alpha_1$, and $s_2 = \alpha_2/u$. Therefore $s_0 \cdot s_2 = \alpha_0\alpha_2$, which is invariant across all positive steady states for all parameter values, a second substrate robustness property besides ACR in $s_1$. As $T_s \to \infty$, the parameter $u \to \infty$, so $s_0 \to \infty$ while $s_2 \to 0$. Thus the stable positive steady state asymptotically approaches a \emph{hypophosphorylated state} in which essentially all substrate is unphosphorylated ($S_0$), with a small fixed amount partially phosphorylated ($S_1 = \alpha_1$) and vanishingly little fully phosphorylated ($S_2 \to 0$). The stable dead boundary steady state, by contrast, has all substrate in the fully phosphorylated form $S_2$, which we call the \emph{fully phosphorylated state}.

The network is thus bistable between these two qualitatively extreme states for all $T_s$ above the saddle-node threshold. As $T_s$ increases, ACR forces $s_1$ to remain fixed while the stable positive steady state shifts ever further toward the hypophosphorylated extreme, and the basin of attraction of the fully phosphorylated state shrinks. The network therefore operates as a switch: initial conditions determine which state the system settles into, with the hypophosphorylated state becoming the increasingly dominant attractor as total substrate grows. In this sense the network acts as a phosphorylation resistor: the more substrate is added, the more strongly the system defaults to the hypophosphorylated state.
\end{remark}

\subsection{The $(E,E,C_1,C_2)$ network}

\begin{figure}[h!]
    \centering
\includegraphics[width=1\linewidth]{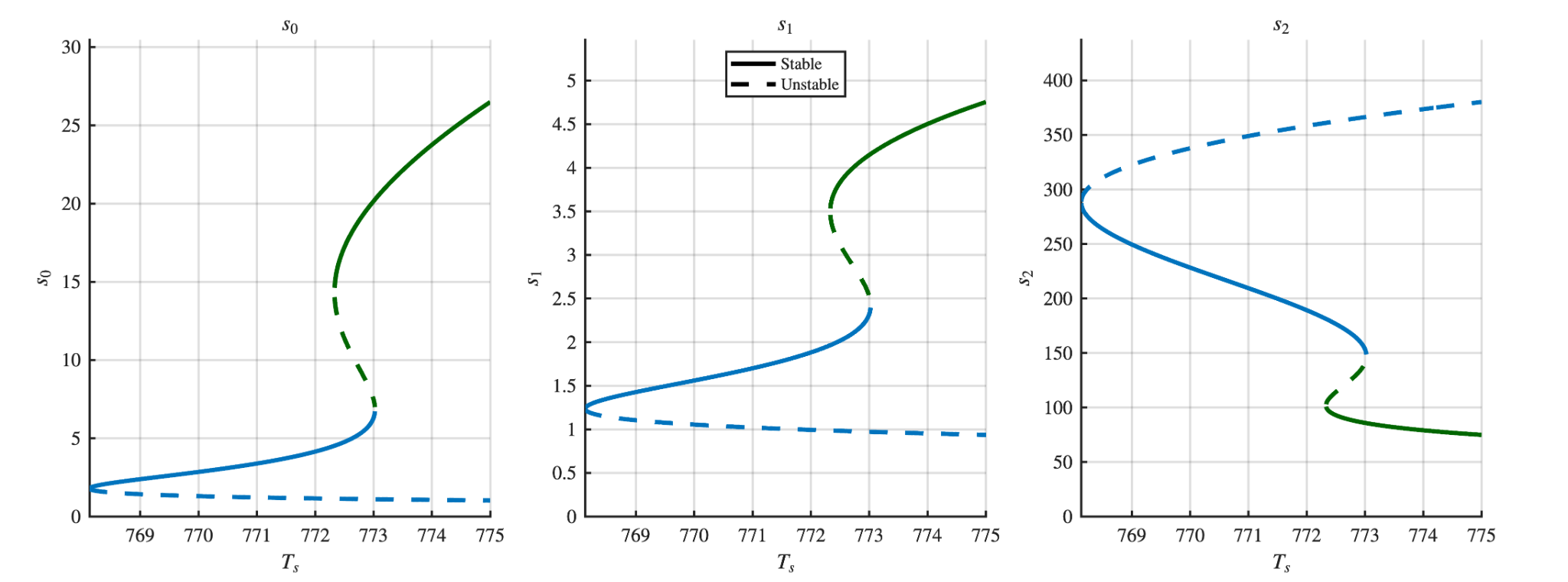}
    \caption{Bifurcation diagram for the $(E,E,C_1,C_2)$ network showing $s_0$, $s_1$, $s_2$ and multistability (two simultaneously stable positive steady states) out of a maximum of four. The dead boundary steady state, which is stable for all values of $T_s$, is omitted. A pair of positive steady states (one stable, one unstable) emerges at each of two saddle-node bifurcations, giving four positive steady states over an interval of $T_s$. The two distinct positive steady state branches are shown in blue and dark green to distinguish them; solid lines indicate stable steady states and dashed lines indicate unstable steady states. There is no ACR in any species. Parameter values: $k_1^+ = 1$, $k_1^- = 1$, $k_1 = 0.036$, $k_2^+ = 1$, $k_2^- = 11.13$, $k_2 = 1$, $k_3^+ = 1$, $k_3^- = 4886.7$, $k_3 = 1$, $k_4^+ = 1$, $k_4^- = 1$, $k_4 = 0.976$, $T_e = 661$.}
    \label{fig:eec1c2}
\end{figure}

In this subsection, we study the dynamics of the $(E,E,C_1,C_2)$ network. Since the enzymes of the reverse chain are $\EE_2=C_1$ and $\EE_1=C_2$, we have $v_1=u$ and $v_2=1/u$ (thus $w_2=1$). At positive steady states, equation \eqref{eq:mainpoly_u} becomes $P_{C_1,C_2}(u)=0$ where

\begin{equation*}
\begin{aligned}
P_{C_1,C_2}(u)=u^{5}a_4\alpha_{0}
+ u^{4}\!\left(a_3\alpha_{0} + a_4\alpha_{1}\right)
+ u^{3}\!\left(
b_{4}T_{e}
+ \gamma\alpha_{0}
+ a_3\alpha_{1}
- a_4T_{s}
\right) \\
+ u^{2}\!\left(
b_{3}T_{e}
+ \gamma\alpha_{1}
+ a_4\alpha_{2}
- a_3T_{s}
\right)
+ u\!\left(
a_3\alpha_{2}
-\gamma T_{s}
\right)
+ \gamma\alpha_{2}
.    
\end{aligned}
\end{equation*}

\begin{theorem}\label{thm:EEC1C2}
    Consider the model $(E,E,C_1,C_2)$. 
    \begin{enumerate}
        \item (Boundary steady state) For any choice of positive rate constants and $T_e,T_s>0$, there is a unique boundary steady state with $e=T_e$ and $s_2=T_s$. This boundary steady state is locally stable for any choice of positive rate constants and $T_e,T_s>0$.
        \item (Positive steady state existence) For any choice of positive rate constants and any $T_e>0$,  a positive steady state exists for any $T_s > T_s^{1,2}$, where
        \[
        T_s^{1,2} = \alpha_0+\alpha_1+\alpha_2+\frac{T_e(b_4+b_3)}{a_4+a_3+\gamma}.
        \]    
        Furthermore, the number of positive steady states (including multiplicity) is even. 
        \item (Positive steady state capacity) For any choice of positive rate constants and $T_e, T_s>0$, there are at most four positive steady states.  Furthermore, there exists a choice of positive rate constants and $T_e, T_s>0$ for which there are four positive steady states.
        \item (Multistability) There exists a choice of positive rate constants and $T_e, T_s >0$ for which there are two positive stable steady states. 
        
    \end{enumerate}
\end{theorem}
\begin{proof}
The proofs of part 1 follows from the next-generation matrix method. The detailed calculations can be found in the Supplementary Material.

For part 2, we observe that $P_{C_1,C_2}(0)=(a_3+\gamma)\alpha_2>0$ and $P_{C_1,C_2}(u)\sim u^3a_4\alpha_0 >0$ for sufficiently large $u$. Thus the number of positive roots of $P_{C_1,C_2}(u)$ (and the number of positive steady states) must be even. To show that a positive root exists for large enough $T_s$, we calculate
\begin{align*}
P_{C_1,C_2}(1)=(a_4+a_3+\gamma)(\alpha_0+\alpha_1+\alpha_2)+T_e(b_4+b_3)-(a_4+a_3+\gamma)T_s=(a_4+a_3+\gamma)(T_s^{1,2}-T_s).
\end{align*}
Thus $P_{C_1,C_2}(1)<0$ for any $T_s>T_s^{1,2}$. As a result, $P_{C_1,C_2}(u)$ has a positive root by the intermediate value theorem, and thus a positive steady state exists by Theorem \ref{thm:1-1}.

For part 3 and 4, we first note that $P_{C_1,C_2}(u)$ has at most four sign changes, thus it has at most four positive roots. We are able to find a choice of positive rate constants and total concentrations such that $P_{C_1,C_2}(u)$ has exactly four positive roots, thus the system has four positive steady states. The choice of rate constants and total concentrations is given as follows: $k_1^+ = 645$, $k_1^- = 2.2$, $k_1 = 2.2$, $k_2^+ = 76.8$, $k_2^- = 4$, $k_2 = 4$, $k_3^+ = 8$, $k_3^- = 1$, $k_3 = 1$, $k_4^+ = 51$, $k_4^- = 10.87$, $k_4 = 10.87$, $T_e = 12.45$, and $T_s=772.5$ (see Figure \ref{fig:eec1c2}). Furthermore,  after substituting these parameter values into the steady-state parameterization, we computed the reduced Jacobian at each of the four positive steady states and evaluated its eigenvalues in Mathematica; two of the four steady states have all nonzero eigenvalues with negative real part, so there are two stable positive steady states with this choice of rate constants and total concentrations.

\end{proof}

\begin{remark}
Although the $(E,E,C_1,C_2)$ network does not have ACR in any species (Theorem~\ref{thm:ACR}), the product $s_1 s_2$ is nonetheless robust across positive steady states: from the parameterization \eqref{eq:parameters_cu}, at every positive steady state we have $s_1 s_2 = \alpha_1\alpha_2 v_1 v_2 = \alpha_1\alpha_2$, a quantity depending only on the rate constants. The parameterization also gives $s_1 = \alpha_1 u$ and $s_2 = \alpha_2/u$, so $s_1$ and $s_2$ move in opposite directions as $u$ varies. For the specific parameter values shown in Figure~\ref{fig:eec1c2}, the stable positive steady state branches have increasing $u$ as $T_s$ increases, giving $s_1$ increasing and $s_2$ decreasing with $T_s$ — a negative dose-response in the final output. Whether this holds for all parameter values depends on which branch is stable and how $u$ varies along it, and is not established in general.
\end{remark}

\begin{remark}[Necessary condition for four positive steady states]
We observe that the maximum number of sign changes of $P_{C_1,C_2}(u)$ is achieved when the coefficients of $u^3$ and $u$ are negative while that of $u^2$ is positive. In particular, the condition for this is given by: 
\begin{equation*}
        \max \left \{ \frac{a_3\alpha_2}{\gamma}, \frac{1}{a_4}(b_4T_e + \gamma\alpha_0+a_3\alpha_1) \right\} < T_s 
        < 
        \frac{1}{a_3}(b_3T_e + \gamma\alpha_1+a_4\alpha_2),
 \end{equation*}       
 or equivalently
 \begin{align*}
        &\max \left \{ \frac{k_1 k_2^*}{k_3^* k_4^*}\left(1+\frac{k_2}{k_3}\right),
        \left[k_1\left(\frac{1}{k_1^*} + \frac{1}{k_4^*}\left(1+\frac{k_2}{k_3}\right)\right) + \left(1 + \frac{2k_1}{k_4}\right) T_e\right]\left(1+ \frac{k_1}{k_4}\right)^{-1}
        \right\}
        \\
        &\quad \quad < T_s <
        \left[k_2\left(\frac{1}{k_2^*} + \frac{1}{k_3^*}\left(1+\frac{k_1}{k_4}\right)\right) + \left(1 + \frac{2k_2}{k_3}\right) T_e\right]\left(1+ \frac{k_2}{k_3}\right)^{-1}.
\end{align*}    
\end{remark}

Figure \ref{fig:eec1c2} illustrates the dynamics of the $(E,E,C_1,C_2)$ network for a specific choice of rate constants and total enzyme $T_e$. The  boundary steady state, which is stable for any value of $T_s$, is omitted from the figure. We observe a small interval of $T_s$ where the mass-action system admits four positive steady states, two of which are stable. As a result, the system exhibits bistability.

\subsection{The $(E,E,C_2,C_2)$ network}

\begin{figure}[h!]
    \centering
\includegraphics[width=1\linewidth]{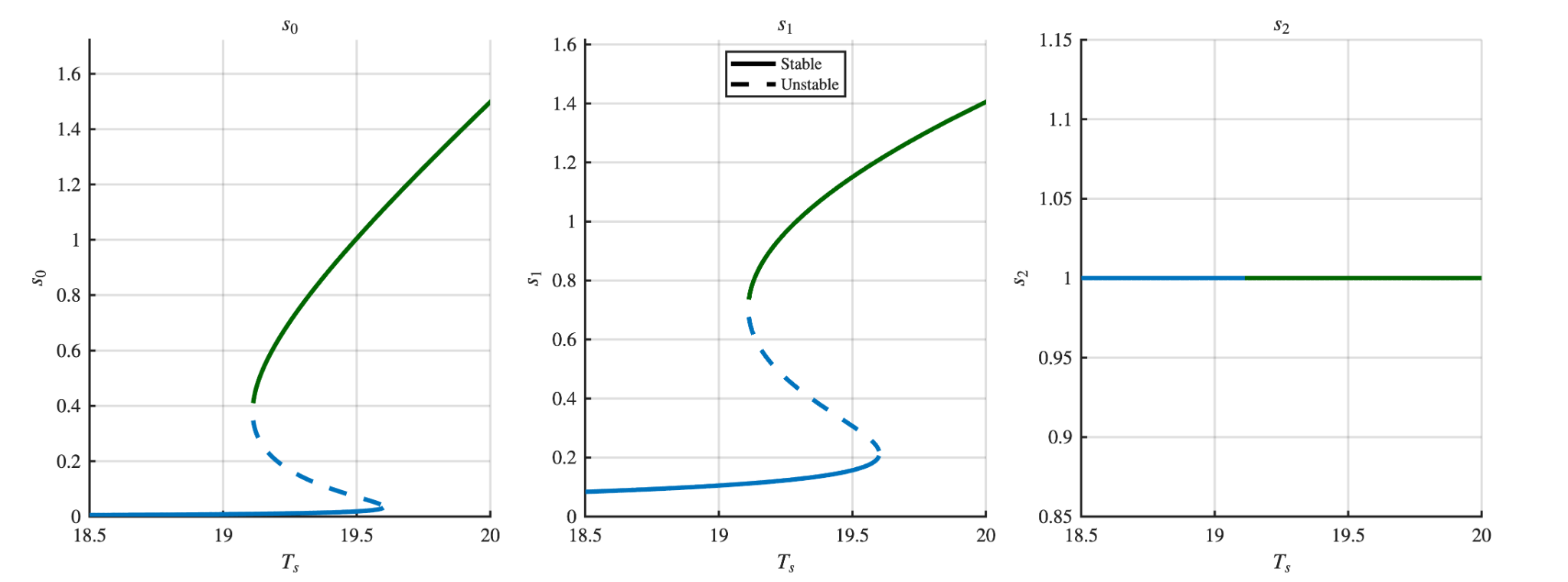}
    \caption{Bifurcation diagram for the $(E,E,C_2,C_2)$ network showing $s_0$, $s_1$, $s_2$ and multistability (two simultaneously stable positive steady states) together with ACR in species $S_2$. The dead boundary steady state branch is omitted; it is stable for $T_s < k_2/k_3^*$ and loses stability at the transcritical bifurcation where the first positive steady state emerges. A stable and an unstable positive steady state are subsequently born at a saddle-node bifurcation, giving three positive steady states over an interval of $T_s$. The two distinct positive steady state branches are shown in blue and dark green to distinguish them; solid lines indicate stable steady states and dashed lines indicate unstable steady states. ACR in $S_2$ is visible as a flat branch in the $s_2$ panel. Parameter values: $k_1^+ = 645$, $k_1^- = 2.2$, $k_1 = 2.2$, $k_2^+ = 76.8$, $k_2^- = 4$, $k_2 = 4$, $k_3^+ = 8$, $k_3^- = 1$, $k_3 = 1$, $k_4^+ = 51$, $k_4^- = 10.87$, $k_4 = 10.87$, $T_e = 12.45$.}
    \label{fig:eec2c2}
\end{figure}

In this subsection, we study the dynamics of the $(E,E,C_2,C_2)$ network. Since the enzymes of the reverse chain are $\EE_2=\EE_1=C_2$, we have $v_1=u$ and $v_2=1$ (thus $w_2=u$). At positive steady states, equation \eqref{eq:mainpoly_u} becomes $R(u)=0$ where
\begin{align*}
R(u):=\frac{P_{C_2,C_2}(u)}{u}=u^{4}a_4\alpha_{0}
+& u^{3}\!\left(a_3\alpha_{0} + a_4\alpha_{1}\right)
+ u^{2}\!\left(
b_{4}T_{e}
+ \gamma\alpha_{0}
+ a_3\alpha_{1}
+ a_4(\alpha_{2}
- T_{s})
\right)\\
&+ u\!\left(
b_{3}T_{e}
+ \gamma\alpha_{1}
+ a_3(\alpha_{2}
- T_{s})
\right)
+  \gamma\left(
\alpha_{2}
-T_{s}
\right).
\end{align*}

\begin{theorem}\label{thm:EEC2C2}
    Consider the model $(E,E,C_2,C_2)$. 
    \begin{enumerate}
        \item (Boundary steady state) For any choice of positive rate constants and $T_e,T_s>0$, there is a unique boundary steady state with $e=T_e$ and $s_2=T_s$. This steady state is locally stable if $T_s<k_2/k_3^*$ and unstable if $T_s>k_2/k_3^*$.
        \item (Positive steady state existence) \cite{joshi2024bifunctional} For any choice of positive rate constants and any $T_e>0$, a positive steady state exists if and only if $T_s>k_2/k_3^*$. Furthermore, the number of positive steady states (including multiplicity) is odd. 
        \item (Positive steady state capacity) For any choice of positive rate constants and $T_e, T_s>0$, there are at most three positive steady states.  Furthermore, there exists a choice of positive rate constants and $T_e, T_s>0$ for which there are three positive steady states. 
        \item (Multistability) There exists a choice of positive rate constants and $T_e, T_s >0$ for which there are two positive stable steady states. 
    \end{enumerate}
\end{theorem}
    
\begin{proof}
The proof of part 1 follows from the next-generation matrix method; the detailed calculations are in the Supplementary Material.

For part 2, we observe that when $T_s\leq k_2/k_3^* = \alpha_2$, $R(u)$ has no sign change.
When $T_s>k_2/k_3^*$, we have $R(0)=\gamma(\alpha_2-T_s)<0$ and $R(u)\sim u^4a_4\alpha_0>0$ for sufficiently large $u$.  Continuity therefore guarantees at least one positive root. Moreover, because the signs at $u=0$ and for sufficiently large positive $u$ are opposite, the total number of positive roots counted with multiplicity is odd.
In particular, $R(u)$ has at least one zero.
The statement of part 2 now follows by Theorem \ref{thm:1-1}.

For part 3, since $R(u)$ has at most three sign changes, there are at most three positive steady states. We are able to find a choice of positive rate constants and total concentrations such that $R(u)$ has exactly three positive roots, which implies the system has exactly three positive steady states. The choice of rate constants and total concentrations is given as follows: $k_1^+ = 645$, $k_1^- = 2.2$, $k_1 = 2.2$, $k_2^+ = 76.8$, $k_2^- = 4$, $k_2 = 4$, $k_3^+ = 8$, $k_3^- = 1$, $k_3 = 1$, $k_4^+ = 51$, $k_4^- = 10.87$, $k_4 = 10.87$, $T_e = 12.45$, and $T_s=19.5$.

For part 4, by computing the eigenvalues of the Jacobian matrix at the steady states found for the parameters in the previous paragraph, we showed that there are two stable positive steady states.

\end{proof}

Figure \ref{fig:eec2c2} illustrates the dynamics of the $(E,E,C_2,C_2)$ network for a specific choice of rate constants and total enzyme $T_e$. As suggested by Theorem \ref{thm:EEC2C2}, there is a transcritical bifurcation at $T_s=k_2/k_3^*$ where the boundary steady state becomes unstable and a positive steady state emerges and is stable. Since this behavior is similar to that in Figure \ref{fig:eec2c1_1pss}, we omit it from Figure \ref{fig:eec2c2} and focus only on the new dynamical behaviors that occur at larger values of $T_s$. In particular, there is an interval of $T_s$ in which the system admits three positive steady states, two of which are stable, exhibiting bistability. The $(E,E,C_2,C_2)$ network is therefore the only one among the four that simultaneously exhibits bistability and ACR.

\section{Discussion}
\label{sec:discussion}

\begin{table}[t]
\centering
\begin{tabular}{llllllll}
\toprule
\multirow{2}{*}{Model} & \multicolumn{2}{c}{Dead boundary} & \multicolumn{2}{c}{Living boundary} & \multicolumn{3}{c}{Positive steady state} \\ \cmidrule(lr){2-3} \cmidrule(lr){4-5} \cmidrule(lr){6-8} 
 & Existence & Stability & Existence & Stability & Existence & Max & ACR \\ 
\midrule
$(E,E,F,F)$ & None & N/A & None & N/A & All & 3 & None \\
\midrule
$(E,E,C_2,C_1)$ & All & $T_s < \frac{k_2}{k_3^*}$ & $T_s > \frac{k_2}{k_3^*}$ & $T_s<T_s^{2,1}$ & $T_s > T_s^{2,1}$ & $2$ & $s_{1} = \frac{k_1}{k_4^*},\; s_{2} = \frac{k_2}{k_3^*}$  \\
\addlinespace
$(E,E,C_1,C_1)$ & All & All & None & N/A & $T_s > T_s^{1,1}$ & $2$ & $s_1= \frac{k_1}{k_4^*}$  \\
\addlinespace
$(E,E,C_1,C_2)$ & All & All  & None & N/A & $T_s>T_s^{1,2}$ & $4$ & None \\
\addlinespace
$(E,E,C_2,C_2)$ & All & $T_s < \frac{k_2}{k_3^*}$ & None & N/A & $T_s > \frac{k_2}{k_3^*}$ & $3$ & $s_{2} = \frac{k_2}{k_3^*}$ \\
\bottomrule
\end{tabular}
\caption{Summary of Theorems \ref{thm:EEC2C1}, \ref{thm:EEC1C1}, \ref{thm:EEC1C2}, and \ref{thm:EEC2C2} for the dual-site futile cycles with bifunctional enzymes, with the non-bifunctional network $(E,E,F,F)$ included for comparison. For boundary steady states, we show conditions on $T_s$ under which the steady state exists and is stable (for any choice of positive rate constants and $T_e > 0$). For positive steady states, we show conditions for existence and the maximum number possible in any stoichiometric compatibility class. In the Existence and Stability columns, ``All'' means for all positive rate constants and all $T_e, T_s > 0$, and ``None'' means for no choice of positive rate constants and $T_e, T_s > 0$. The ACR column shows the steady-state value of the ACR species, if any. The thresholds $T_s^{1,1}$, $T_s^{1,2}$, $T_s^{2,1}$ are explicit functions of the rate constants and $T_e$; see the respective theorems for their definitions.}
\label{table:summary}
\end{table}

Table~\ref{table:summary} summarizes the dynamical properties of the five networks studied in this paper. In this section we discuss the broader significance of these results, comparing the four bifunctional enzyme networks to each other and to the non-bifunctional case, and drawing connections to biological function.

\subsection{Mathematical implications of bifunctional architecture}

\paragraph{ACR from the substrate hypergraph.}
A key contribution of \cite{joshi2026bifunctional} is the identification of the \emph{substrate hypergraph} as the underlying structural cause of ACR in enzymatic networks. The substrate hypergraph tracks only which substrate species are transformed in each reaction, abstracting away all mechanistic details. The central result is that ACR is a property of the hypergraph alone: if the appropriate structural condition is satisfied, then ACR holds for \emph{any} detailed model associated with that hypergraph, regardless of which specific mechanism is assumed. In particular, Theorems~\ref{thm:ACR_1site} and~\ref{thm:ACR} hold for any valid detailed model of the respective substrate hypergraphs, not only for the Henri-Michaelis-Menten mechanism studied here.

\paragraph{Boundary steady states: a hallmark of bifunctional enzymes.}
A clear distinction between the bifunctional enzyme networks and the conventional two-enzyme case $(E,E,F,F)$ is the presence of boundary steady states. In the $(E,E,F,F)$ network, a positive steady state exists for all parameter values, and there are no boundary steady states. By contrast, in all four bifunctional enzyme networks a dead boundary steady state always exists, corresponding to all substrate locked in the fully phosphorylated form $S_2$ with the enzyme $E$ free.

The $(E,E,C_2,C_1)$ network additionally admits a living boundary steady state when $T_s>k_2/k_3^*$, which is stable for intermediate values of $T_s$.

The stability of the dead boundary steady state also varies considerably across networks. In $(E,E,C_2,C_1)$ and $(E,E,C_2,C_2)$, the dead boundary steady state is stable only when $T_s < k_2/k_3^*$ and becomes unstable above this threshold, giving rise to a transcritical bifurcation. In $(E,E,C_1,C_1)$ and $(E,E,C_1,C_2)$, the dead boundary steady state is stable for all parameter values, and positive steady states emerge through a saddle-node bifurcation rather than a transcritical one.

\paragraph{Positive steady states: capacity and bifurcation structure.}
The four bifunctional networks exhibit a wide range of behaviors in their positive steady states. The maximum number of positive steady states ranges from two (in $(E,E,C_2,C_1)$ and $(E,E,C_1,C_1)$) to three (in $(E,E,C_2,C_2)$) to four (in $(E,E,C_1,C_2)$), compared to a maximum of three for the non-bifunctional $(E,E,F,F)$ network. The bifurcation structure connecting boundary and positive steady states also varies. In $(E,E,C_2,C_1)$ and $(E,E,C_2,C_2)$, the positive steady state branch is connected to the dead boundary steady state through a transcritical bifurcation. 
In $(E,E,C_1,C_1)$ and $(E,E,C_1,C_2)$, no such transcritical bifurcation occurs; instead, a pair of positive steady states (one stable, one unstable) is born simultaneously in a saddle-node bifurcation at a threshold $T_s$ that depends on the rate constants and $T_e$.

In the $(E,E,C_2,C_1)$ network, the connection between boundary and positive steady state branches is particularly rich. As $T_s$ increases, the living boundary steady state first gains stability through a transcritical bifurcation with the dead boundary steady state, then loses it through a second transcritical bifurcation with the positive steady state branch, or for certain parameter choices through a saddle-node bifurcation that creates the positive steady state below the transcritical threshold. This latter scenario constitutes a \emph{backward bifurcation} \cite{gumel2012causes}: the positive steady state appears before the living boundary steady state becomes unstable, creating an overlap region where both are stable simultaneously.

Taken together, these results show that changing which enzyme--substrate compound carries out the reverse steps can alter not only the number of steady states but also the way stable and unstable branches are connected as total substrate varies.

\subsection{Biological implications of bifunctional architecture}
The mathematical results illustrate several possible functional consequences of bifunctional architecture. ACR can buffer selected modification levels against changes in total protein abundance, while bistability and hysteresis can support history-dependent responses. Boundary steady states introduce another possibility, in which part of the modification cycle is inactive while another part remains active.

\paragraph{Boundary steady states and partial activity.}
At the living boundary steady state of the $(E,E,C_2,C_1)$ network, the subsystem $S_1 \rightleftarrows S_2$ remains active while the reactions involving $S_0$ are inactive. This differs qualitatively from the dead boundary steady state, at which all reaction rates vanish. Thus bifunctionality can permit steady activity in part of the modification cycle even when another part is inactive.

\paragraph{Hysteresis and history dependence.}
Backward bifurcations are associated with \emph{hysteresis} \cite{strogatz2024nonlinear}: as $T_s$ is increased and then decreased, the system may follow different branches, settling into the positive steady state on the way up but remaining at the living boundary steady state on the way down. Thus, over a range of total substrate values, the long-term response can depend on the parameter history rather than only on the current value of $T_s$, providing a possible mechanism for memory-like behavior in a signaling system.

\paragraph{Simultaneous robustness and flexibility.}
The four bifunctional enzyme networks display different combinations of multistationarity, multistability, and ACR.

The $(E,E,C_2,C_1)$ and $(E,E,C_1,C_1)$ networks each admit multiple positive steady states within the same stoichiometric compatibility class (i.e., for the same initial total substrate and enzyme), and each has ACR in at least one species. However, neither network exhibits bistability between positive steady states.

The $(E,E,C_1,C_2)$ network achieves bistability. For certain parameter values, two positive steady states are simultaneously stable within the same compatibility class but the network has no ACR in any individual species. The stable steady states differ in all of $s_0$, $s_1$, and $s_2$, so the network can switch between two quantitatively different internal states, but neither state is distinguished by a robust output.

The $(E,E,C_2,C_2)$ network achieves both simultaneously. It is bistable, admitting two simultaneously stable positive steady states within the same compatibility class, and it has ACR in $S_2$. This means that the two stable steady states share the same value of $s_2$, yet differ in $s_0$ and $s_1$. Thus, the distribution among modification levels can differ between two stable regimes without changing the final modification level. This combination is not present in the non-bifunctional comparison network $(E,E,F,F)$.

\paragraph{Non-positive dose-response as a signature of bifunctional enzyme architecture.}
The ACR and product robustness properties of the bifunctional enzyme networks directly constrain the dose-response of $S_2$ with respect to total substrate $T_s$. Networks with ACR in $S_2$ exhibit a rigorously zero dose-response for all parameter values. For the $(E,E,C_1,C_2)$ and $(E,E,C_1,C_1)$ networks, numerical evidence and the parameterization suggest a negative dose-response in $S_2$, and we conjecture this holds for all parameter values; a key step toward a rigorous proof would be establishing that the positive steady state with the larger value of $u$ is always locally asymptotically stable, which we intend to pursue in future work. For the parameter regimes examined here, we do not observe a positive dose-response of $S_2$ in the bifunctional networks, in contrast to the non-bifunctional case where a biphasic response has been shown to arise from substrate sequestration effects \cite{suwanmajo2013biphasic}. The bifunctional architecture may therefore alter the dose-response landscape, replacing biphasic behavior with zero or purely negative responses in the examples studied here. Because the negative dose-response statements for $(E,E,C_1,C_1)$ and $(E,E,C_1,C_2)$ are presently supported by the parameterization and numerical evidence rather than a general proof, we treat this pattern as a conjectural direction for future work rather than as an established network-wide theorem.

The above discussion takes $S_2$ as the network output. The bifunctional architecture is inherently directional: $S_0$ appears only as the tail of a $c$-edge in the substrate hypergraph and is therefore structurally excluded from having ACR, while $S_2$ and $S_1$ appear as tails of $e$-edges and can have ACR \cite{joshi2026bifunctional}. Correspondingly, $S_0$ increases with $T_s$ in the cases studied, while $S_2$ is constant in the ACR cases and decreases in the parameter regimes examined for the remaining bifunctional networks. These behaviors contrast with the biphasic response possible in the non-bifunctional symmetric setting.

\paragraph{Bistability without bound: a network operating at two extremes.}
As established in Remark~\ref{rem:eec1c1_asymptotics}, the $(E,E,C_1,C_1)$ network is bistable between two qualitatively extreme phosphorylation states 
for all $T_s$ above the saddle-node threshold. This can be considered \emph{asymptotic bistability} if we take into consideration both boundary and positive steady states. In particular, 
both stable steady states
persist as $T_s\to\infty$.

The two stable regimes correspond to qualitatively opposite phosphorylation patterns: the dead boundary steady state is fully phosphorylated, whereas along the stable positive branch the system becomes increasingly hypophosphorylated as $T_s$ grows, with $s_2\to 0$. Thus these contrasting phosphorylation patterns coexist for arbitrarily large total substrate. The non-bifunctional comparison network $(E,E,F,F)$ has no boundary steady states, so this particular mechanism for persistent bistability is unavailable there.

\appendix

\section*{Appendix: Mathematical Background and Further Reading}

This appendix provides a non-technical overview of the key mathematical concepts used in this paper, along with an annotated list of references for readers seeking deeper background on either the biological or mathematical sides.

\paragraph{Chemical reaction networks and mass-action kinetics.}
A chemical reaction network is a collection of species and reactions, each specifying which species are consumed and produced. Under mass-action kinetics, the rate of each reaction is proportional to the product of the concentrations of the reacting species. This yields a system of polynomial ordinary differential equations (ODEs) governing how concentrations evolve over time. The study of such systems (their steady states, stability, and dependence on parameters) is the central mathematical task of this paper.

\paragraph{Steady states and bifurcations.}
A steady state is a concentration vector at which all ODEs are zero, so concentrations do not change. A steady state is stable if small perturbations decay back to it, and unstable if perturbations grow away. A bifurcation occurs when, as a parameter changes, the number or stability of steady states changes qualitatively. Key bifurcation types in this paper are: transcritical bifurcations (where two steady state branches cross and exchange stability), saddle-node bifurcations (where a pair of steady states is born or destroyed), and backward bifurcations (a special transcritical case associated with hysteresis).

\paragraph{Boundary steady states and the next-generation matrix method.}
A boundary steady state is one where some species concentrations are zero. Assessing the stability of boundary steady states requires care, since the standard Jacobian eigenvalue test includes directions that point outside the biologically relevant region. The next-generation matrix (NGM) method, adapted from mathematical epidemiology, handles this cleanly: it constructs a non-negative matrix from the linearization at the boundary steady state and computes its spectral radius $\rho$. If $\rho < 1$ the boundary steady state is stable (nearby trajectories return to it); if $\rho > 1$ it is unstable.

\paragraph{Absolute concentration robustness and the substrate hypergraph.}
A species has absolute concentration robustness (ACR) if its steady-state concentration is the same at every positive steady state, for any choice of total conserved quantities. This means the output is insensitive to changes in total substrate or enzyme levels. The substrate hypergraph is a graphical representation of a reaction network that records only which substrates are transformed in each reaction step, abstracting away all mechanistic details. A key result used in this paper is that ACR can be detected directly from the structure of the substrate hypergraph, independently of the specific kinetic mechanism.

\subsection*{Further Reading}

\subsubsection*{Biological background}
\begin{itemize}
    \item \textbf{Multisite phosphorylation and signal transduction:} \cite{markevich2004signaling, wang2008singularly, suwanmajo2013biphasic}. These cover how protein phosphorylation cascades transduce and amplify cellular signals, the role of multisite modification in generating switch-like and bistable responses, and how distributive multisite phosphorylation with separate enzymes can exhibit biphasic dose-response behavior.
    \item \textbf{Bifunctional enzymes and bacterial signaling:} \cite{shinar2009robustness, russo1993essential, batchelor2003robustness}. These describe the biology of two-component signaling systems in bacteria, including the EnvZ/OmpR and IDHKP-IDH examples discussed in the paper.
    \item \textbf{Bistability and cellular memory:} \cite{ferrell1996tripping}. An accessible account of how bistability arises in cell signaling and its functional consequences for cellular decision-making.
\end{itemize}

\subsubsection*{Mathematical background}
\begin{itemize}
    \item \textbf{Chemical reaction network theory:} \cite{feinberg2019foundations}. A comprehensive and accessible textbook introduction to the mathematics of reaction networks, covering mass-action kinetics, steady states, and stability theory.
    \item \textbf{Absolute concentration robustness:} \cite{shinar2010structural, joshi2026bifunctional}. The original paper introducing ACR and the structural conditions for it in deficiency-one networks (Shinar--Feinberg), and the substrate hypergraph approach extending ACR theory to general enzymatic networks (Joshi--Nguyen).
    \item \textbf{Next-generation matrix method:} \cite{van2002reproduction, johnston2026boundary}. The original epidemiological paper introducing the NGM and basic reproduction number (van den Driessche--Watmough), and its adaptation to boundary steady-state stability in biochemical reaction networks (Johnston).
    \item \textbf{Bifurcation theory and nonlinear dynamics:} \cite{strogatz2024nonlinear}. An excellent introductory textbook covering bifurcations, bistability, hysteresis, and related dynamical phenomena at an accessible level.
    \item \textbf{Multistationarity in phosphorylation networks:} \cite{wang2008singularly, markevich2004signaling}. These papers establish that multisite phosphorylation cycles can admit multiple positive steady states, providing the mathematical and biological context for the present work.
\end{itemize}

\section*{Acknowledgements}

BJ is grateful for support from AMS-Simons Research Enhancement Grants for Primarily Undergraduate Institution (PUI). MDJ is supported by National Science Foundation Grant No. DMS-2213390. MDJ, BJ, and TDN gratefully acknowledge the University of Wisconsin-Madison and the 4th Workshop on Mathematics of Reaction Networks for facilitating research discussions.

\section*{Author Contributions}

All authors contributed to the conception and design of the study, the mathematical analysis, and the writing and revision of the manuscript.

\section*{Conflict of Interest Statement}

The authors declare no conflicts of interest.

\section*{Data and Code Availability}

No experimental data were generated or used in this study. Mathematica was used to produce the data on steady state coordinates and stability. MATLAB was used to produce the bifurcation diagrams in Figures~3--8. Maple was used to perform the next generation matrix calculations. All code is made publicly available at \url{https://github.com/TungDaoNguyen/dual-site-futile-cycles}.

\section*{AI Use Disclosure}

Claude (Anthropic, claude.ai) was used to assist with writing Mathematica and MATLAB code for numerical computations and generation of bifurcation diagrams. Claude was also used to assist with editing and polishing the written text for clarity, grammar, and style. Google Gemini was used to convert matrices from Maple and format them in \LaTeX. ChatGPT (OpenAI) was used during revision to assist with editing, reorganizing, and polishing the manuscript. All mathematical content, results, and conclusions are the sole work of the authors.

\clearpage
\section*{Supplementary Material}
\setcounter{section}{0}
\renewcommand{\thesection}{\arabic{section}}

This document contains detailed algebraic computations and proofs supporting the main paper. Section~1 presents the full mass-action system, steady-state parameterization, proofs, and stability analysis (including Routh array computations and explicit steady-state coordinates) for the single-site network $(E,C_1)$. Section~2 contains the next-generation matrix calculations establishing boundary steady-state stability for all four dual-site networks.

\section{Single-Site Futile Cycle with Bifunctional Enzyme}
\label{sec:appendix_1site}

We present here the full details of the single-site futile cycle with a bifunctional enzyme introduced in Section~3 of the main paper, including the mass-action system, steady state parameterization, and proofs of Theorem~3.1 and Theorem~3.2.

The single-site substrate modification network in Figure Figure~2 of the main paper corresponds to the following mass-action system
\begin{equation}\label{eq:ODE_1site}\left\{ \; \;  \begin{aligned}
\frac{ds_0}{dt} &= -k_1^+ s_0e + k_1^- c_1 + k_2 c_2, \quad & & \frac{ds_1}{dt} = k_1 c_1 - k_2^+ s_1c_1 + k_2^- c_2, \\
\frac{de}{dt}   &= -k_1^+ s_0e + (k_1^- + k_1) c_1, \\
\frac{dc_1}{dt} &= k_1^+ s_0e - (k_1^- + k_1) c_1 - k_2^+ s_1c_1 + (k_2^- + k_2) c_2,  \quad & & \frac{dc_2}{dt} = k_2^+ s_1c_1 - (k_2^- + k_2) c_2,
\end{aligned}\right.\end{equation}
which has conservation laws
\begin{equation}\label{eq:cons_1site}
    \begin{aligned}
T_s & = s_0 + s_1 + c_1 + 2c_2, \\
T_e & = e + c_1 + c_2.
\end{aligned}
\end{equation}

We obtain a steady state parameterization of the mass-action system \eqref{eq:ODE_1site} based on the approach in \cite{joshi2026bifunctional}.
The mass-action system \eqref{eq:ODE_1site} can be written as:
\begin{equation}\label{eq:ODE_1site_GL}\left\{ \; \;  \begin{aligned}
\frac{ds_0}{dt} &= L_0-G_1, \quad \frac{ds_1}{dt} = L_1-G_2, \\
\frac{de}{dt}   &= -G_1, \\
\frac{dc_1}{dt} &= G_1-G_2, \quad \frac{dc_2}{dt} = G_2.
\end{aligned}\right.\end{equation}
where 
\[
G_1= k_1^+ s_0e - (k_1^- + k_1) c_1 \quad \text{and}\quad G_2=k_2^+ s_1c_1 - (k_2^- + k_2) c_2
\]
represent the enzyme-holding or compound-producing rates of edge 1 $(S_0\to S_1)$ and edge 2 $(S_1\to S_0)$ respectively. The expressions
\[
L_0 = k_2 c_2 - k_1c_1\quad \text{and} \quad L_1=k_1c_1-k_2c_2
\]
represent the currents into a substrate minus the currents out of the substrate (for more details on the concept of currents on a substrate hypergraph, see \cite{joshi2026bifunctional}). It is important to note that since $C_1$ is the intermediate compound in edge 1 and the enzyme in edge 2, the equation of $c_1$ is the sum of its contributions in both edges: $\frac{dc_1}{dt}=G_1-G_2$.

\begin{lemma}[\cite{joshi2026bifunctional}, Lemma S12.3]\label{lem:GL}
    The steady state equations for \eqref{eq:ODE_1site_GL} are equivalent to
\begin{align*}
    &\text{(edge conditions)} \quad G_\ell=0 \quad \text{for}\quad \ell=1,2\\
    &\text{(node conditions)}\quad L_i=0 \quad \text{for} \quad i=0,1.
\end{align*}
\end{lemma}
For convenience, we recall that $k_\ell^* \coloneqq k_\ell k_\ell^+ / (k_\ell + k_\ell^-)$ (Notation equation~(2.3) of the main paper).
The next lemma follows directly from Lemma Lemma~S1.1.
\begin{lemma}\label{lem:SSP_1site}
    All steady states of \eqref{eq:ODE_1site} must satisfy:
    \[
    k_1^*s_0 e= k_1c_1 = k_2^*s_1c_1 = k_2c_2.  
    \]
\end{lemma}
Lemma Lemma~S1.2 allows us to directly obtain the following parameterization.
\begin{proposition}\label{prop:PSSP_1site}
We define $u:=c_1/e$. The positive steady states of \eqref{eq:ODE_1site} can be parameterized in terms of $e$ and $u$ as follows
\[
s_0 = \frac{k_1}{k_1^*} u, \quad s_1 = \frac{k_1}{k_2^*}, \quad c_1 = ue, \quad c_2 = \frac{k_1}{k_2} ue.
\]
In particular, the model \eqref{eq:ODE_1site} has ACR in species $S_1$ with ACR value $k_1/k_2^*$ (Theorem~Theorem~3.3).
\end{proposition}

We now present proofs of Theorem~3.1 and Theorem~3.2.
\begin{proof}[Proof of Theorem~3.1]
From Lemma Lemma~S1.1, at any boundary steady state we must have $c_1=c_2=0$. Thus from the conservation laws \eqref{eq:cons_1site}, we have $e=T_e>0$ at any boundary steady state. Since $s_0e=0$, we must have $s_0=0$ and $s_1=T_s$. Therefore, there is a unique boundary steady state with $e=T_e$ and $s_1=T_s$. It is easy to verify that at this boundary steady state, all reaction rates are zero, which implies the steady state is a dead steady state.

To examine the stability of the boundary steady state, we utilize the next-generation matrix method introduced in Section~S4. On the species which go to zero, we have the reduced system:
\[ \left\{ \; \; \;
\begin{aligned}
\frac{d s_0}{dt} & = -k_1^+ s_0e + k_1^-c_1 + k_2c_2\\
\frac{d c_1}{dt} & = k_1^+s_0e - (k_1^- + k_1)c_1 - k_2^+s_1c_1 + (k_2^- + k_2)c_2\\
\frac{d c_2}{dt} & = k_2^+s_1c_1 - (k_2^{-} + k_2) c_2
\end{aligned}
\right. \]
We utilize the following decomposition $f(x) = \mathcal{F} - \mathcal{V}$:
\[
\begin{aligned}
\mathcal{F} & = \langle k_2 c_2, 0, 0 \rangle,\\
\mathcal{V} & = \langle k_1^{+} s_0 e - k_1^{-} c_1,\; 
-k_1^{+} s_0 e + (k_1^{-} + k_1) c_1 + k_2^{+} s_1 c_1 - (k_2^{-} + k_2) c_2,\; 
 - k_2^{+} s_1 c_1 + (k_2^{-} + k_2) c_2 \rangle.
\end{aligned}
\]
We take the Jacobians of $\mathcal{F}$ and $\mathcal{V}$ with respect to $\{ s_0, c_1, c_2 \}$ and evaluate them at the boundary steady state $x^*$ to obtain
\[
F = \begin{bmatrix}
0 & 0 & k_2 \\
0 & 0 & 0 \\
0 & 0 & 0 
\end{bmatrix}
\quad \text{and} \quad
V = \begin{bmatrix}
k_1^{+} T_e & -k_1^{-} & 0 \\
-k_1^{+} T_e & k_2^{+} T_s + k_1^{-} + k_1 & -k_2^{-} - k_2 \\
0 & -k_2^{+} T_s & k_2^{-} + k_2
\end{bmatrix}.
\]
We now compute
\[
V^{-1} = \left[ \begin{array}{ccc} 
\displaystyle \frac{1}{k_1^* T_e} 
& \displaystyle \frac{k_1^{-}}{k_1^{+} k_1 T_e} 
& \displaystyle \frac{k_1^{-}}{k_1^{+} k_1 T_e} \\[1em]
\displaystyle \frac{1}{k_1} 
& \displaystyle \frac{1}{k_1} 
& \displaystyle \frac{1}{k_1} \\[1em]
\displaystyle \frac{k_2^{*} T_s}{k_1 k_2} 
& \displaystyle \frac{k_2^{*} T_s}{k_1 k_2} 
& \displaystyle \frac{k_2^*T_s}{k_1 k_2} + \frac{k_2^*}{k_2^+k_2}
\end{array} \right]
\]
so that
\[
FV^{-1} = 
\left[ \begin{array}{ccc}
\displaystyle \frac{k_2^{*} T_s}{k_1} 
& \displaystyle \frac{k_2^{*} T_s}{k_1} 
& \displaystyle \frac{k_2^*T_s}{k_1} + \frac{k_2^*}{k_2^+} \\
0 & 0 & 0 \\
0 & 0 & 0
\end{array} \right].
\]
We have
\[
\rho(FV^{-1}) 
= \frac{k_2^*}{k_1} T_s.
\]
It follows by Theorem~2.4 that $x^*$ is locally asymptotically stable for every stoichiometric compatibility class satisfying $\displaystyle T_s < k_1/k_2^*$ ($\rho(FV^{-1}) < 1$) and unstable for every compatibility class satisfying $\displaystyle T_s >
 k_1/k_2^*$ ($\rho(FV^{-1}) > 1$).
\end{proof}

\begin{proof}[Proof of Theorem~3.2]

Plugging the steady state parameterization in Proposition Theorem~3.1 into the enzyme conservation $T_e = e+c_1+c_2$ and substrate conservation $T_s = s_0+s_1+c_1+2c_2$ yields

\begin{align}
\label{eq:ec1_cons}
    T_e = e \left(1+ \left(1 + \frac{k_1}{k_2} \right)u \right), \quad \quad 
    T_s = \frac{k_1}{k_1^*} u + \frac{k_1}{k_2^*} + \left( 1 + \frac{2k_1}{k_2}\right) ue. 
\end{align}
Eliminating $e$ in \eqref{eq:ec1_cons} gives us 
\begin{align}
\label{eq:e_u}
    e = \frac{T_e}{\left(1+ \left(1 + \frac{k_1}{k_2} \right)u \right)}. 
\end{align}
Substituting \eqref{eq:e_u} into \eqref{eq:ec1_cons} leads to a single equation in terms of the variable $u$: 
\begin{align*} \label{eq:Ts_func_u_ec1}
    T_s = \frac{k_1}{k_1^*} u + \frac{k_1}{k_2^*} + \frac{\left( 1 + \frac{2k_1}{k_2}\right) uT_e}{1+ \left(1 + \frac{k_1}{k_2} \right) u},
\end{align*}
which is equivalent to the following polynomial equation in $u$: 
\begin{equation}\label{eq:u_1site}
    \frac{k_1}{k_1^*} \left(1 + \frac{k_1}{k_2} \right) u^2 + \left[ \frac{k_1}{k_1^*} + \left( 1 + \frac{2k_1}{k_2}\right) T_e + \left(1 + \frac{k_1}{k_2} \right) \left(\frac{k_1}{k_2^*} - T_s\right)\right] u + \left(\frac{k_1}{k_2^*} - T_s\right) = 0. 
\end{equation}
From \eqref{eq:e_u} and the parameterization in Proposition Theorem~3.1, there is a correspondence between the positive roots of \eqref{eq:u_1site} and the positive steady states of the model \eqref{eq:ODE_1site}. It is straightforward to see that \eqref{eq:u_1site} has a positive root if and only if $T_s > k_1/k_2^*$. 
As a result, there is a unique positive steady state in every compatibility class with $T_s > k_1/k_2^*$. 

To show the local stability of this unique positive steady state, we use the Routh-Hurwitz stability criterion \cite{hurwitz1964conditions}. 
Using the two conservation laws, we eliminate two species $e$ and $s_0$. Using the parameterization in Proposition Theorem~3.1, the Jacobian matrix of the three-dimensional system in the variables $c_1, c_2$ and $s_1$ can be written in terms of the rate constants, $T_e$, and a single concentration-related variable $u$:
\begin{equation}
\displaystyle
J=\begin{bmatrix}
-k_1^- (1 + u) - \frac{k_2 k_1^+ T_e}{k_2 + (k_1 + k_2) u} - \frac{k_1 (k_2^- + k_2 (2 + u))}{k_2} & k_2 + k_2^- - (k_1 + k_1^-) u - \frac{2 k_2 k_1^+ T_e}{k_2 + (k_1 + k_2) u} & -\frac{k_2 T_e (k_1^+ + k_2^+ u)}{k_2 + (k_1 + k_2) u} \\
\frac{k_1 (k_2 + k_2^-)}{k_2} & -k_2 - k_2^- & \frac{k_2 k_2^+ T_e u}{k_2 + (k_1 + k_2) u} \\
-\frac{k_1 k_2^-}{k_2} & k_2^- & -\frac{k_2 k_2^+ T_e u}{k_2 + (k_1 + k_2) u}
\end{bmatrix}
\end{equation}
The characteristic polynomial can be written as
\[
p(\lambda) = \frac{a_0 + a_1 \lambda + a_2 \lambda^2 + a_3 \lambda^3}{k_2 (k_2 + (k_1 + k_2) u)^2}, 
\]
where the coefficients, computed using Mathematica, are given by:
\begin{equation} \label{eq:coeffscharpoly_ec1}
    \begin{aligned}
a_0 &= k_2^2 k_2^+ T_e u (k_2 (k_2 (k_1 + k_1^-) + (2 k_1 + k_2) k_1^+ T_e) \\
&\quad + 2 k_2 (k_1 + k_2) (k_1 + k_1^-) u + (k_1 + k_2)^2 (k_1 + k_1^-) u^2), \\
a_1 &= k_1^4 (k_2 + k_2^-) u^3 + k_1^3 (k_2 + k_2^-) u^2 (k_1^- u + 3 k_2 (1 + u)) \\
&\quad + k_1^2 k_2 u (k_2^- k_1^+ T_e + 3 k_1^- k_2^- u (1 + u) + 3 k_2^2 (1 + u)^2 \\
&\quad + k_2 (2 k_1^+ T_e + 3 k_1^- u (1 + u) + 3 k_2^- (1 + u)^2 + k_2^+ T_e u (2 + u))) \\
&\quad + k_2^3 ((k_2^- + k_2^- u + k_2^+ T_e u) (k_1^+ T_e + k_1^- (1 + u)^2) \\
&\quad + k_2 (1 + u) (k_1^- (1 + u)^2 + T_e (k_1^+ + k_2^+ u))) \\
&\quad + k_1 k_2^2 (k_1^- k_2^+ T_e u^2 (1 + u) + 3 k_1^- k_2^- u (1 + u)^2 + k_2^2 (1 + u)^3 \\
&\quad + k_2^- k_1^+ (T_e + 2 T_e u) + k_2 (3 k_1^- u (1 + u)^2 + k_2^- (1 + u)^3 \\
&\quad + k_1^+ T_e (2 + 3 u) + k_2^+ T_e u (2 + u (4 + u)))), \\
a_2 &= (k_2 + (k_1 + k_2) u) (k_2^2 (k_2 + k_1^- + k_2^- + k_1^+ T_e \\
&\quad + (k_2 + 2 k_1^- + k_2^- + k_2^+ T_e) u + k_1^- u^2) \\
&\quad + k_1^2 u (k_2^- + k_2 (2 + u)) \\
&\quad + k_1 k_2 (k_2^- + u (k_1^- + 2 k_2^- + k_1^- u) + k_2 (2 + u (4 + u)))), \\
a_3 &= k_2 (k_2 + (k_1 + k_2) u)^2.
\end{aligned}
\end{equation}
Using the Routh-Hurwitz criterion, to prove the local stability of the positive steady state, it suffices to show that $a_i>0$ for $i=0,1,2,3$ and that $a_1a_2>a_0a_3$. We verify these conditions directly with Mathematica.
\end{proof}

The explicit coordinates of the unique positive steady state can be derived as follows.

\medskip
\noindent\textit{Explicit steady state coordinates.} Define the dimensionless parameters:
\begin{equation}
\kappa_2 = \frac{k_1}{k_2}, \quad \kappa_2^* =   T_s - \frac{k_1}{k_2^*}, \quad \kappa_1^* = \frac{k_1}{k_1^*}.
\end{equation}

A necessary and sufficient condition for the existence of a positive zero (and therefore a positive steady state) is that $T_s >k_1/k_2^*$, and so $\kappa_2^*$ is greater than zero at any positive steady state.

The discriminant is given by:
\begin{equation}
\Delta = \left(\kappa_1^* - \kappa_2^* (1 + \kappa_2) + (1 + 2 \kappa_2) T_e\right)^2 + 4 \kappa_1^* \kappa_2^* (1 + \kappa_2) .
\end{equation}

The positive solution for $u$ is:
\begin{equation}
u = \frac{\sqrt{\Delta}  + \kappa_2^* (1 + \kappa_2) - \kappa_1^* - (1 + 2 \kappa_2) T_e }{2 \kappa_1^* (1 + \kappa_2)}.
\end{equation}

Multiplying by $\kappa_1^*$ gives the $s_0$ coordinate: $s_0 = \kappa_1^* u$.

\section{Next-Generation Matrix Method}\label{sec:appendix_bss}

In this Appendix, we prove the boundary stability results from Section~4 of the main paper using the next-generation matrix method \cite{johnston2026boundary}.

\begin{proof}[Proof of Theorem~4.5 part $1.$, $(E,E,C_2,C_1)$ network]
Consider the system of differential equations \eqref{eq:c2c1de}:
\begin{equation}
\label{eq:c2c1de}
\left\{ \; \; \;
\begin{aligned}
\frac{d s_0}{dt} &= -k_1^{+} s_0 e + k_1^{-} c_1 + k_4 c_4, \\
\frac{d s_1}{dt} &= -k_4^{+} s_1 c_1 - k_2^{+} s_1 e + k_1 c_1 + k_2^{-} c_2 + k_3 c_3 + k_4^{-} c_4, \\
\frac{d s_2}{dt} &= -k_3^{+} s_2 c_2 + k_2 c_2 + k_3^{-} c_3, \\
\frac{d e}{dt} &= -k_1^{+} s_0 e + (k_1^{-} + k_1) c_1 - k_2^{+} s_1 e + (k_2^{-} + k_2) c_2, \\
\frac{d c_1}{dt} &= k_1^{+} s_0 e - (k_1^{-} + k_1) c_1 - k_4^{+} s_1 c_1 + (k_4^{-} + k_4) c_4, \\
\frac{d c_2}{dt} &= k_2^{+} s_1 e - (k_2^{-} + k_2) c_2 - k_3^{+} s_2 c_2 + (k_3^{-} + k_3) c_3, \\
\frac{d c_3}{dt} &= k_3^{+} s_2 c_2 - (k_3^{-} + k_3) c_3, \\
\frac{d c_4}{dt} &= k_4^{+} s_1 c_1 - (k_4^{-} + k_4) c_4.
\end{aligned}
\right.
\end{equation}
This system has the dead boundary steady state $(s_0, s_1,  s_2, e, c_1, c_2, c_3, c_4) = (0, 0, T_s, T_e, 0,0,0,0)$ and the living boundary steady state
\begin{equation}
\label{eq:living}
(s_0, s_1,  s_2, e, c_1, c_2, c_3, c_4) = \left(0, s_1, \frac{k_2}{k_3^*}, e, 0, \frac{e \, s_1 \, k_2^*}{k_2}, \frac{e \, s_1 \, k_2^* }{k_3}, 0\right).
\end{equation}
where $s_1$ and $e$ are functions of $T_s$ and $T_e$. 

The system \eqref{eq:c2c1de} restricted to the vanishing species of the dead boundary steady state is
\[
\left\{ \; \; \;
\begin{aligned}
\frac{d s_0}{dt} &= -k_1^{+} s_0 e + k_1^{-} c_1 + k_4 c_4, \\
\frac{d s_1}{dt} &= -k_4^{+} s_1 c_1 - k_2^{+} s_1 e + k_1 c_1 + k_2^{-} c_2 + k_3 c_3 + k_4^{-} c_4, \\
\frac{d c_1}{dt} &= k_1^{+} s_0 e - (k_1^{-} + k_1) c_1 - k_4^{+} s_1 c_1 + (k_4^{-} + k_4) c_4, \\
\frac{d c_2}{dt} &= k_2^{+} s_1 e - (k_2^{-} + k_2) c_2 - k_3^{+} s_2 c_2 + (k_3^{-} + k_3) c_3, \\
\frac{d c_3}{dt} &= k_3^{+} s_2 c_2 - (k_3^{-} + k_3) c_3, \\
\frac{d c_4}{dt} &= k_4^{+} s_1 c_1 - (k_4^{-} + k_4) c_4.
\end{aligned}
\right.
\]
We utilize the following decomposition $f(x) = \mathcal{F} - \mathcal{V}$ of \eqref{eq:c2c1de}:
\[\mathcal{F} = \langle 0, k_3 c_3, 0, 0, 0, 0 \rangle\]
and
\[\mathcal{V} = \begin{bmatrix}
k_1^+ s_0 e - k_1^- c_1 - k_4 c_4 \\
k_4^+ s_1 c_1 + k_2^+ s_1 e - k_1 c_1 - k_2^- c_2 - k_4^- c_4 \\
-k_1^+ s_0 e + (k_1^- + k_1) c_1 + k_4^+ s_1 c_1 - (k_4^- + k_4) c_4 \\
-k_2^+ s_1 e + (k_2^- + k_2) c_2 + k_3^+ s_2 c_2 - (k_3^- + k_3) c_3 \\
-k_3^+ s_2 c_2 + (k_3^- + k_3) c_3 \\ 
-k_4^+ s_1 c_1 + (k_4^- + k_4) c_4
\end{bmatrix}
\]
We take the Jacobians of $\mathcal{F}$ and $\mathcal{V}$ with respect to $\{ s_0, s_1, c_1, c_2, c_3, c_4 \}$ and evaluate them at the dead boundary steady state $x^*$ to obtain
\[
F = \begin{bmatrix}
0 & 0 & 0 & 0 & 0 & 0 \\
0 & 0 & 0 & 0 & k_3 & 0 \\
0 & 0 & 0 & 0 & 0 & 0 \\
0 & 0 & 0 & 0 & 0 & 0 \\
0 & 0 & 0 & 0 & 0 & 0 \\
0 & 0 & 0 & 0 & 0 & 0
\end{bmatrix}
\quad \text{and} \quad
V = \begin{bmatrix}
k_1^+ T_e & 0 & -k_1^- & 0 & 0 & -k_4 \\
0 & k_2^+ T_e & -k_1 & -k_2^- & 0 & -k_4^- \\
-k_1^+ T_e & 0 & k_1^- + k_1 & 0 & 0 & -k_4^- - k_4 \\
0 & -k_2^+ T_e & 0 & k_3^+ T_s + k_2 + k_2^- & -k_3^- - k_3 & 0 \\
0 & 0 & 0 & -k_3^+ T_s & k_3^- + k_3 & 0 \\
0 & 0 & 0 & 0 & 0 & k_4^- + k_4
\end{bmatrix}
\]
We now compute
\[
V^{-1} = \begin{bmatrix}
\frac{1}{k_1^* T_e} & 0 & \frac{k_1^-}{k_1 k_1^+ T_e} & 0 & 0 & \frac{k_4^*}{k_1^*k_4^+T_e}+\frac{k_1^-}{k_1k_1^+T_e}\\
\frac{1}{k_2^* T_e} & \frac{1}{k_2^* T_e} & \frac{1}{k_2^* T_e} & \frac{k_2^-}{k_2 k_2^+ T_e} & \frac{k_2^-}{k_2 k_2^+ T_e} & \frac{2}{k_2^* T_e} \\
\frac{1}{k_1} & 0 & \frac{1}{k_1} & 0 & 0 & \frac{1}{k_1} + \frac{k_4^*}{k_1k_4^+} \\
\frac{1}{k_2} & \frac{1}{k_2} & \frac{1}{k_2} & \frac{1}{k_2} & \frac{1}{k_2} & \frac{2}{k_2} \\
\frac{k_3^* T_s}{k_2 k_3} & \frac{k_3^* T_s}{k_2 k_3} & \frac{k_3^* T_s}{k_2 k_3} & \frac{k_3^* T_s}{k_2 k_3} & \frac{k_3^*T_s}{k_2 k_3} + \frac{k_3^*}{k_3^+k_3} & \frac{2k_3^*T_s}{k_2 k_3} \\
0 & 0 & 0 & 0 & 0 & \frac{k_4^*}{k_4 k_4^+}
\end{bmatrix}
\]
so that
\[
FV^{-1} = \begin{bmatrix}
0 & 0 & 0 & 0 & 0 & 0 \\
\frac{k_3^* T_s}{k_2} & \frac{k_3^* T_s}{k_2} & \frac{k_3^* T_s}{k_2} & \frac{k_3^* T_s}{k_2} & \frac{k_3^*T_s}{k_2} + \frac{k_3^*}{k_3^+} &\frac{2k_3^* T_s}{k_2} \\
0 & 0 & 0 & 0 & 0 & 0 \\
0 & 0 & 0 & 0 & 0 & 0 \\
0 & 0 & 0 & 0 & 0 & 0 \\
0 & 0 & 0 & 0 & 0 & 0
\end{bmatrix}
\]
We have
\[
\rho(FV^{-1}) = \frac{k_3^* T_s}{k_2}
\]
It follows by Theorem~2.4 that $x^*$ is locally asymptotically stable for every stoichiometric compatibility class satisfying $T_s < k_2/k_3^*$ and unstable for every compatibility class satisfying $T_s > k_2/k_3^*$.

The system \eqref{eq:c2c1de} restricted to the vanishing species of the living boundary steady state is:
\[
\left\{ \; \; \;
\begin{aligned}
\frac{d s_0}{dt} &= -k_1^{+} s_0 e + k_1^{-} c_1 + k_4 c_4, \\
\frac{d c_1}{dt} &= k_1^{+} s_0 e - (k_1^{-} + k_1) c_1 - k_4^{+} s_1 c_1 + (k_4^{-} + k_4) c_4, \\
\frac{d c_4}{dt} &= k_4^{+} s_1 c_1 - (k_4^{-} + k_4) c_4.
\end{aligned}
\right.
\]
We utilize the following decomposition $f(x) = \mathcal{F} - \mathcal{V}$ of \eqref{eq:c2c1de}:
\[\mathcal{F} = \langle k_4 c_4, 0, 0 \rangle\]
and
\[\mathcal{V} = \left\langle 
k_1^+ s_0 e - k_1^- c_1, \quad 
-k_1^+ s_0 e + (k_1^- + k_1) c_1 + k_4^+ s_1 c_1 - (k_4^- + k_4) c_4, \quad 
-k_4^+ s_1 c_1 + (k_4^- + k_4) c_4 
\right\rangle\]
We take the Jacobians of $\mathcal{F}$ and $\mathcal{V}$ with respect to $\{ s_0, c_1, c_4 \}$ and evaluate them at the living boundary state to obtain
\[
F = \begin{bmatrix}
0 & 0 & k_4 \\
0 & 0 & 0 \\
0 & 0 & 0
\end{bmatrix}
\quad \text{and} \quad
V = \begin{bmatrix}
k_1^+ e & -k_1^- & 0 \\
-k_1^+ e & k_4^+ s_1 + k_1 + k_1^- & -k_4^- - k_4 \\
0 & -k_4^+ s_1 & k_4^- + k_4
\end{bmatrix}
\]
We now compute
\[
V^{-1} = \begin{bmatrix}
\frac{1}{k_1^* e} & \frac{k_1^-}{k_1 k_1^+ e} & \frac{k_1^-}{k_1 k_1^+ e} \\
\frac{1}{k_1} & \frac{1}{k_1} & \frac{1}{k_1} \\
\frac{k_4^* s_1}{k_1 k_4} & \frac{k_4^* s_1}{k_1 k_4} & \frac{k_4^* s_1}{k_1 k_4}+\frac{k_4^*}{k_4k_4^+}
\end{bmatrix}
 \; \; \; \mbox{and} \; \; \;
FV^{-1} =\begin{bmatrix}
\frac{k_4^* s_1}{k_1} & \frac{k_4^* s_1}{k_1} & \frac{k_4^* s_1}{k_1}+\frac{k_4^*}{k_4^+} \\
0 & 0 & 0 \\
0 & 0 & 0
\end{bmatrix}.\]
We have
\[
\rho(FV^{-1}) = \frac{k_4^* s_1}{k_1}
\]

We now consider properties of the living boundary steady state. Note that the subsystem of \eqref{eq:c2c1de} resulting from taking $s_0=c_1=c_4=0$ is invariant. Furthermore, it corresponds precisely to the system \eqref{eq:ODE_1site} after re-indexing the substrates ($S_1 \mapsto S_0$, $S_2 \mapsto S_1$), intermediates ($C_2 \mapsto C_1$, $C_3 \mapsto C_2$), and rate constants ($k_2^{\pm} \mapsto k_1^{\pm}$, $k_3^{\pm} \mapsto k_2^{\pm}$). Furthermore, substituting $s_0=c_1=c_4=0$ into the conservation laws of the main paper yields
\begin{align}
\label{eq:cons1}
T_s & = s_1 + s_2 + c_2 + 2c_3, \\
\label{eq:cons2}
T_e & = e + c_2 + c_3.
\end{align}
which maps precisely to the conservation laws \eqref{eq:cons_1site} with the reindexing. Consequently, we may apply Theorem~3.2 to this subsystem. We have that there is a positive steady state of the subsystem if and only if $T_s > k_2/k_3^*$ and that this steady state is unique and locally asymptotically stable relative to the subsystem whenever it exists. Extending to the full system \eqref{eq:c2c1de}, this steady state corresponds to the living boundary steady state of \eqref{eq:c2c1de}. It follows that the living boundary steady state exists if $T_s > k_2/k_3^*$. Furthermore, since the living boundary steady state is locally asymptotically stable with respect to the intersection of the boundary face and compatibility class and $\rho(FV^{-1}) = k_4^* s_1/k_1$, it follows from Theorem~2.4 that it is stable if $s_1 < k_1/k_4^*$ and unstable if $s_1 > k_1/k_4^*$, where $s_1$ is the $s_1$ entry of the living boundary steady state, and we are done.\\

\end{proof}

\begin{proof}[Proof of Theorem~4.7 part $1.$, $(E,E,C_1,C_1)$ network]
Consider the system of differential equations \eqref{eq:c1c1de}:
\begin{equation}
\label{eq:c1c1de}
\left\{ \; \; \;
\begin{aligned}
\frac{d s_0}{dt} &= -k_1^{+} s_0 e + k_1^{-} c_1 + k_4 c_4, \\
\frac{d s_1}{dt} &= -k_4^{+} s_1 c_1 - k_2^{+} s_1 e + k_1 c_1 + k_2^{-} c_2 + k_3 c_3 + k_4^{-} c_4, \\
\frac{d s_2}{dt} &= -k_3^{+} s_2 c_1 + k_2 c_2 + k_3^{-} c_3, \\
\frac{d e}{dt} &= -k_1^{+} s_0 e + (k_1^{-} + k_1) c_1 - k_2^{+} s_1 e + (k_2^{-} + k_2) c_2, \\
\frac{d c_1}{dt} &= k_1^{+} s_0 e - (k_1^{-} + k_1) c_1 - k_3^{+} s_2 c_1 + (k_3^{-} + k_3) c_3 \\
&\qquad {} - k_4^{+} s_1 c_1 + (k_4^{-} + k_4) c_4, \\
\frac{d c_2}{dt} &= k_2^{+} s_1 e - (k_2^{-} + k_2) c_2, \\
\frac{d c_3}{dt} &= k_3^{+} s_2 c_1 - (k_3^{-} + k_3) c_3, \\
\frac{d c_4}{dt} &= k_4^{+} s_1 c_1 - (k_4^{-} + k_4) c_4.
\end{aligned}
\right.
\end{equation}
This system has the boundary steady states $(s_0, s_1,  s_2, e, c_1, c_2, c_3, c_4) = (0, 0, T_s, T_e, 0,0,0,0)$. The system \eqref{eq:c1c1de} restricted to the vanishing species of the dead boundary steady state is
\[
\left\{ \; \; \;
\begin{aligned}
\frac{d s_0}{dt} &= -k_1^{+} s_0 e + k_1^{-} c_1 + k_4 c_4, \\
\frac{d s_1}{dt} &= -k_4^{+} s_1 c_1 - k_2^{+} s_1 e + k_1 c_1 + k_2^{-} c_2 + k_3 c_3 + k_4^{-} c_4, \\
\frac{d c_1}{dt} &= k_1^{+} s_0 e - (k_1^{-} + k_1) c_1 - k_3^{+} s_2 c_1 + (k_3^{-} + k_3) c_3 \\
&\qquad {} - k_4^{+} s_1 c_1 + (k_4^{-} + k_4) c_4, \\
\frac{d c_2}{dt} &= k_2^{+} s_1 e - (k_2^{-} + k_2) c_2, \\
\frac{d c_3}{dt} &= k_3^{+} s_2 c_1 - (k_3^{-} + k_3) c_3, \\
\frac{d c_4}{dt} &= k_4^{+} s_1 c_1 - (k_4^{-} + k_4) c_4.
\end{aligned}
\right.
\]
We utilize the following decomposition $f(x) = \mathcal{F} - \mathcal{V}$ of \eqref{eq:c1c1de}:
\[\mathcal{F} = \langle 0, k_3 c_3, 0, 0, 0, 0 \rangle\]
and
\[
\mathcal{V} = \begin{bmatrix}
\begin{aligned}
& k_1^+ s_0 e - k_1^- c_1 - k_4 c_4 \\
& k_4^+ s_1 c_1 + k_2^+ s_1 e - k_1 c_1 - k_2^- c_2 - k_4^- c_4 \\
& - k_1^+ s_0 e + (k_1^- + k_1) c_1 + k_3^+ s_2 c_1 - (k_3^- + k_3) c_3 + k_4^+ s_1 c_1 - (k_4^- + k_4) c_4 \\
& - k_2^+ s_1 e + (k_2^- + k_2) c_2 \\
& - k_3^+ s_2 c_1 + (k_3^- + k_3) c_3 \\
& - k_4^+ s_1 c_1 + (k_4^- + k_4) c_4
\end{aligned}
\end{bmatrix}
\]

We take the Jacobians of $\mathcal{F}$ and $\mathcal{V}$ with respect to $\{ s_0, s_1, c_1, c_2, c_3, c_4 \}$ and evaluate them at the boundary steady state $x^*$ to obtain
\[
F = \begin{bmatrix}
0 & 0 & 0 & 0 & 0 & 0 \\
0 & 0 & 0 & 0 & k_3 & 0 \\
0 & 0 & 0 & 0 & 0 & 0 \\
0 & 0 & 0 & 0 & 0 & 0 \\
0 & 0 & 0 & 0 & 0 & 0 \\
0 & 0 & 0 & 0 & 0 & 0
\end{bmatrix} \quad \text{and} \quad V = \begin{bmatrix}
k_1^+ T_e & 0 & -k_1^- & 0 & 0 & -k_4 \\
0 & k_2^+ T_e & -k_1 & -k_2^- & 0 & -k_4^- \\
-k_1^+ T_e & 0 & k_3^+ T_s + k_1 + k_1^- & 0 & -k_3^- - k_3 & -k_4^- - k_4 \\
0 & -k_2^+ T_e & 0 & k_2^- + k_2 & 0 & 0 \\
0 & 0 & -k_3^+ T_s & 0 & k_3^- + k_3 & 0 \\
0 & 0 & 0 & 0 & 0 & k_4^- + k_4
\end{bmatrix}
\]
We now compute
\[
V^{-1} = \begin{bmatrix}
\frac{1}{k_1^* T_e} & 0 & \frac{k_1^-}{k_1 k_1^+ T_e} & 0 & \frac{k_1^-}{k_1 k_1^+ T_e} & \frac{k_4^*}{k_1^*k_4^+T_e}+\frac{k_1^-}{k_1k_1^+T_e} \\
\frac{1}{k_2^* T_e} & \frac{1}{k_2^* T_e} & \frac{1}{k_2^* T_e} & \frac{1}{k_2^* T_e} & \frac{1}{k_2^* T_e} & \frac{2}{k_2^* T_e} \\
\frac{1}{k_1} & 0 & \frac{1}{k_1} & 0 & \frac{1}{k_1} & \frac{1}{k_1} + \frac{k_4^*}{k_1k_4^+} \\
\frac{1}{k_2} & \frac{1}{k_2} & \frac{1}{k_2} & \frac{1}{k_2} & \frac{1}{k_2} & \frac{2}{k_2} \\
\frac{k_3^* T_s}{k_1 k_3} & 0 & \frac{k_3^* T_s}{k_1 k_3} & 0 & \frac{k_3^*T_s}{k_1 k_3} + \frac{k_3^*}{k_3^+k_3} & \frac{k_3^*T_s}{k_1k_3}+\frac{k_3^*k_4^*T_s}{k_1k_3k_4^+} \\
0 & 0 & 0 & 0 & 0 & \frac{k_4^*}{k_4 k_4^+}
\end{bmatrix}
\]

so that
\[
FV^{-1} = \begin{bmatrix}
0 & 0 & 0 & 0 & 0 & 0 \\
\frac{k_3^* T_s}{k_1} & 0 & \frac{k_3^* T_s}{k_1} & 0 & \frac{k_3^*T_s}{k_1} + \frac{k_3^*}{k_3^+} & \frac{k_3^*T_s}{k_1}+\frac{k_3^*k_4^*T_s}{k_1k_4^+} \\
0 & 0 & 0 & 0 & 0 & 0 \\
0 & 0 & 0 & 0 & 0 & 0 \\
0 & 0 & 0 & 0 & 0 & 0 \\
0 & 0 & 0 & 0 & 0 & 0
\end{bmatrix}
\]
We have
\[
\rho(FV^{-1}) = 0
\]
which is always less than one. It follows by Theorem~2.4 that $x^*$ is locally asymptotically stable for every stoichiometric compatibility class.
\end{proof}

\begin{proof}[Proof of Theorem~4.9 part $1.$, $(E,E,C_1,C_2)$ network]
Consider the system of differential equations \eqref{eq:c1c2de}:
\begin{equation}
\label{eq:c1c2de}
\left\{ \; \; \;
\begin{aligned}
\frac{d s_0}{dt} &= -k_1^{+} s_0 e + k_1^{-} c_1 + k_4 c_4, \\
\frac{d s_1}{dt} &= -k_4^{+} s_1 c_1 - k_2^{+} s_1 e + k_1 c_1 + k_2^{-} c_2 + k_3 c_3 + k_4^{-} c_4, \\
\frac{d s_2}{dt} &= -k_3^{+} s_2 c_2 + k_2 c_2 + k_3^{-} c_3, \\
\frac{d e}{dt} &= -k_1^{+} s_0 e + (k_1^{-} + k_1) c_1 - k_2^{+} s_1 e + (k_2^{-} + k_2) c_2, \\
\frac{d c_1}{dt} &= k_1^{+} s_0 e - (k_1^{-} + k_1) c_1 - k_4^{+} s_1 c_1 + (k_4^{-} + k_4) c_4, \\
\frac{d c_2}{dt} &= k_2^{+} s_1 e - (k_2^{-} + k_2) c_2 - k_3^{+} s_2 c_2 + (k_3^{-} + k_3) c_3, \\
\frac{d c_3}{dt} &= k_3^{+} s_2 c_2 - (k_3^{-} + k_3) c_3, \\
\frac{d c_4}{dt} &= k_4^{+} s_1 c_1 - (k_4^{-} + k_4) c_4.
\end{aligned}
\right.
\end{equation}
This system has the boundary steady states $(s_0, s_1,  s_2, e, c_1, c_2, c_3, c_4) = (0, 0, T_s, T_e, 0,0,0,0)$. The system \eqref{eq:c1c2de} restricted to the vanishing species of the dead boundary steady state is
\[
\left\{ \; \; \;
\begin{aligned}
\frac{d s_0}{dt} &= -k_1^{+} s_0 e + k_1^{-} c_1 + k_4 c_4, \\
\frac{d s_1}{dt} &= -k_4^{+} s_1 c_1 - k_2^{+} s_1 e + k_1 c_1 + k_2^{-} c_2 + k_3 c_3 + k_4^{-} c_4, \\
\frac{d c_1}{dt} &= k_1^{+} s_0 e - (k_1^{-} + k_1) c_1 - k_4^{+} s_1 c_1 + (k_4^{-} + k_4) c_4, \\
\frac{d c_2}{dt} &= k_2^{+} s_1 e - (k_2^{-} + k_2) c_2 - k_3^{+} s_2 c_2 + (k_3^{-} + k_3) c_3, \\
\frac{d c_3}{dt} &= k_3^{+} s_2 c_2 - (k_3^{-} + k_3) c_3, \\
\frac{d c_4}{dt} &= k_4^{+} s_1 c_1 - (k_4^{-} + k_4) c_4.
\end{aligned}
\right.
\]
We utilize the following decomposition $f(x) = \mathcal{F} - \mathcal{V}$ of \eqref{eq:c1c2de}:
\[\mathcal{F} = \langle 0, k_3 c_3, 0, 0, 0, 0 \rangle\]
and
\[
\mathcal{V} = \begin{bmatrix}
\begin{aligned}
& k_1^+ s_0 e - k_1^- c_1 - k_4 c_4 \\
& k_4^+ s_1 c_2 + k_2^+ s_1 e - k_1 c_1 - k_2^- c_2 - k_4^- c_4 \\
& -k_1^+ s_0 e + (k_1^- + k_1) c_1 + k_3^+ s_2 c_1 - (k_3^- + k_3) c_3 \\ 
& -k_2^+ s_1 e + (k_2^- + k_2) c_2 + k_4^+ s_1 c_2 - (k_4^- + k_4) c_4 \\ 
& -k_3^+ s_2 c_1 + (k_3^- + k_3) c_3 \\
& -k_4^+ s_1 c_2 + (k_4^- + k_4) c_4
\end{aligned}
\end{bmatrix}
\]

We take the Jacobians of $\mathcal{F}$ and $\mathcal{V}$ with respect to $\{ s_0, s_1, c_1, c_2, c_3, c_4 \}$ and evaluate them at the boundary steady state $x^*$ to obtain
\[
F = \begin{bmatrix}
0 & 0 & 0 & 0 & 0 & 0 \\
0 & 0 & 0 & 0 & k_3 & 0 \\
0 & 0 & 0 & 0 & 0 & 0 \\
0 & 0 & 0 & 0 & 0 & 0 \\
0 & 0 & 0 & 0 & 0 & 0 \\
0 & 0 & 0 & 0 & 0 & 0
\end{bmatrix} \quad \text{and} \quad V = \begin{bmatrix}
k_1^+ T_e & 0 & -k_1^- & 0 & 0 & -k_4 \\
0 & k_2^+ T_e & -k_1 & -k_2^- & 0 & -k_4^- \\
-k_1^+ T_e & 0 & k_3^+ T_s + k_1 + k_1^- & 0 & -k_3^- - k_3 & 0 \\
0 & -k_2^+ T_e & 0 & k_2^- + k_2 & 0 & -k_4^- - k_4 \\
0 & 0 & -k_3^+ T_s & 0 & k_3^- + k_3 & 0 \\
0 & 0 & 0 & 0 & 0 & k_4^- + k_4
\end{bmatrix}
\]
We now compute
\[
V^{-1} = \begin{bmatrix}
\frac{1}{k_1^* T_e} & 0 & \frac{k_1^-}{k_1 k_1^+ T_e} & 0 & \frac{k_1^-}{k_1 k_1^+ T_e} & \frac{k_4^*}{k_1^* k_4^+ T_e} \\
\frac{1}{k_2^* T_e} & \frac{1}{k_2^* T_e} & \frac{1}{k_2^* T_e} & \frac{k_2^-}{k_2 k_2^+ T_e} & \frac{1}{k_2^* T_e} & \frac{1}{k_2^* T_e} + \frac{k_2^-}{k_2 k_2^+ T_e} \\
\frac{1}{k_1} & 0 & \frac{1}{k_1} & 0 & \frac{1}{k_1} & \frac{k_4^*}{k_1 k_4^+} \\
\frac{1}{k_2} & \frac{1}{k_2} & \frac{1}{k_2} & \frac{1}{k_2} & \frac{1}{k_2} & \frac{2}{k_2} \\
\frac{k_3^* T_s}{k_1 k_3} & 0 & \frac{k_3^* T_s}{k_1 k_3} & 0 & \frac{k_3^*T_s}{k_1 k_3} + \frac{k_3^*}{k_3^+k_3}  & \frac{k_3^*T_s}{k_1k_3}+\frac{k_3^*k_4^*T_s}{k_1k_3k_4^+} \\
0 & 0 & 0 & 0 & 0 & \frac{k_4^*}{k_4 k_4^+}
\end{bmatrix}
.
\]
so that
\[
FV^{-1} = \begin{bmatrix}
0 & 0 & 0 & 0 & 0 & 0 \\
\frac{k_3^* T_s}{k_1} & 0 & \frac{k_3^* T_s}{k_1} & 0 & \frac{k_3^*T_s}{k_1} + \frac{k_3^*}{k_3^+}  & \frac{k_3^*T_s}{k_1}+\frac{k_3^*k_4^*T_s}{k_1k_4^+} \\
0 & 0 & 0 & 0 & 0 & 0 \\
0 & 0 & 0 & 0 & 0 & 0 \\
0 & 0 & 0 & 0 & 0 & 0 \\
0 & 0 & 0 & 0 & 0 & 0
\end{bmatrix}
\]
We have
\[
\rho(FV^{-1}) = 0
\]
which is less than one. It follows by Theorem~2.4 that $x^*$ is locally asymptotically stable for every stoichiometric compatibility class.
\end{proof}

\begin{proof}[Proof of Theorem~4.12 part $1.$, $(E,E,C_2,C_2)$ network]
Consider the system of differential equations \eqref{eq:c2c2de}:
\begin{equation}
\label{eq:c2c2de}
\left\{ \; \; \;
\begin{aligned}
\frac{d s_0}{dt} &= -k_1^{+} s_0 e + k_1^{-} c_1 + k_4 c_4, \\
\frac{d s_1}{dt} &= -k_4^{+} s_1 c_2 - k_2^{+} s_1 e + k_1 c_1 + k_2^{-} c_2 + k_3 c_3 + k_4^{-} c_4, \\
\frac{d s_2}{dt} &= -k_3^{+} s_2 c_2 + k_2 c_2 + k_3^{-} c_3, \\
\frac{d e}{dt} &= -k_1^{+} s_0 e + (k_1^{-} + k_1) c_1 - k_2^{+} s_1 e + (k_2^{-} + k_2) c_2, \\
\frac{d c_1}{dt} &= k_1^{+} s_0 e - (k_1^{-} + k_1) c_1, \\
\frac{d c_2}{dt} &= k_2^{+} s_1 e - (k_2^{-} + k_2) c_2 - k_3^{+} s_2 c_2 + (k_3^{-} + k_3) c_3 \\
&\qquad {} - k_4^{+} s_1 c_2 + (k_4^{-} + k_4) c_4, \\
\frac{d c_3}{dt} &= k_3^{+} s_2 c_2 - (k_3^{-} + k_3) c_3, \\
\frac{d c_4}{dt} &= k_4^{+} s_1 c_2 - (k_4^{-} + k_4) c_4.
\end{aligned}
\right.
\end{equation}
All four networks share the conservation laws (the conservation laws).
This system has the boundary steady states $(s_0, s_1,  s_2, e, c_1, c_2, c_3, c_4) = (0, 0, T_s, T_e, 0,0,0,0)$. The system \eqref{eq:c2c2de} restricted to the vanishing species of the dead boundary steady state is
\[
\left\{ \; \; \;
\begin{aligned}
\frac{d s_0}{dt} &= -k_1^{+} s_0 e + k_1^{-} c_1 + k_4 c_4, \\
\frac{d s_1}{dt} &= -k_4^{+} s_1 c_2 - k_2^{+} s_1 e + k_1 c_1 + k_2^{-} c_2 + k_3 c_3 + k_4^{-} c_4, \\
\frac{d c_1}{dt} &= k_1^{+} s_0 e - (k_1^{-} + k_1) c_1, \\
\frac{d c_2}{dt} &= k_2^{+} s_1 e - (k_2^{-} + k_2) c_2 - k_3^{+} s_2 c_2 + (k_3^{-} + k_3) c_3 \\
&\qquad {} - k_4^{+} s_1 c_2 + (k_4^{-} + k_4) c_4, \\
\frac{d c_3}{dt} &= k_3^{+} s_2 c_2 - (k_3^{-} + k_3) c_3, \\
\frac{d c_4}{dt} &= k_4^{+} s_1 c_2 - (k_4^{-} + k_4) c_4.
\end{aligned}
\right.
\]
We utilize the following decomposition $f(x) = \mathcal{F} - \mathcal{V}$ of \eqref{eq:c2c2de}:
\[\mathcal{F} = \langle 0, k_3 c_3, 0, 0, 0, 0 \rangle\]
and
\[
\mathcal{V} = \begin{bmatrix}
\begin{aligned}
&k_1^+ s_0 e - k_1^- c_1 - k_4 c_4 \\
&k_4^+ s_1 c_2 + k_2^+ s_1 e - k_1 c_1 - k_2^- c_2 - k_4^- c_4 \\
&-k_1^+ s_0 e + (k_1^- + k_1) c_1 \\
&-k_2^+ s_1 e + (k_2^- + k_2) c_2 + k_3^+ s_2 c_2 - (k_3^- + k_3) c_3 + k_4^+ s_1 c_2 - (k_4^- + k_4) c_4 \\
&-k_3^+ s_2 c_2 + (k_3^- + k_3) c_3 \\
&-k_4^+ s_1 c_2 + (k_4^- + k_4) c_4
\end{aligned}
\end{bmatrix}
\]
We take the Jacobians of $\mathcal{F}$ and $\mathcal{V}$ with respect to $\{ s_0, s_1, c_1, c_2, c_3, c_4 \}$ and evaluate them at the boundary steady state $x^*$ to obtain
\[
\begin{bmatrix}
0 & 0 & 0 & 0 & 0 & 0 \\
0 & 0 & 0 & 0 & k_3 & 0 \\
0 & 0 & 0 & 0 & 0 & 0 \\
0 & 0 & 0 & 0 & 0 & 0 \\
0 & 0 & 0 & 0 & 0 & 0 \\
0 & 0 & 0 & 0 & 0 & 0
\end{bmatrix} \quad \text{and} \quad \begin{bmatrix}
k_1^+ T_e & 0 & -k_1^- & 0 & 0 & -k_4 \\
0 & k_2^+ T_e & -k_1 & -k_2^- & 0 & -k_4^- \\
-k_1^+ T_e & 0 & k_1^- + k_1 & 0 & 0 & 0 \\
0 & -k_2^+ T_e & 0 & k_3^+ T_s + k_2 + k_2^- & -k_3^- - k_3 & -k_4^- - k_4 \\
0 & 0 & 0 & -k_3^+ T_s & k_3^- + k_3 & 0 \\
0 & 0 & 0 & 0 & 0 & k_4^- + k_4
\end{bmatrix}
\]
We now compute
\[
V^{-1} = \begin{bmatrix}
\frac{1}{k_1^* T_e} & 0 & \frac{k_1^-}{k_1 k_1^+ T_e} & 0 & 0 & \frac{k_4^*}{k_1^* k_4^+ T_e} \\
\frac{1}{k_2^* T_e} & \frac{1}{k_2^* T_e} & \frac{1}{k_2^* T_e} & \frac{k_2^-}{k_2 k_2^+ T_e} & \frac{k_2^-}{k_2 k_2^+ T_e} & \frac{1}{k_2^* T_e} + \frac{k_2^-}{k_2 k_2^+ T_e} \\
\frac{1}{k_1} & 0 & \frac{1}{k_1} & 0 & 0 & \frac{k_4^*}{k_1 k_4^+}  \\
\frac{1}{k_2} & \frac{1}{k_2} & \frac{1}{k_2} & \frac{1}{k_2} & \frac{1}{k_2} & \frac{2}{k_2} \\
\frac{k_3^* T_s}{k_2 k_3} & \frac{k_3^* T_s}{k_2 k_3} & \frac{k_3^* T_s}{k_2 k_3} & \frac{k_3^* T_s}{k_2 k_3} & \frac{k_3^*T_s}{k_2 k_3} + \frac{k_3^*}{k_3^+k_3}  & \frac{2 k_3^* T_s}{k_2 k_3} \\
0 & 0 & 0 & 0 & 0 & \frac{k_4^*}{k_4 k_4^+} 
\end{bmatrix}
\]
so that
\[
FV^{-1} = \begin{bmatrix}
0 & 0 & 0 & 0 & 0 & 0 \\
\frac{k_3^* T_s}{k_2} & \frac{k_3^* T_s}{k_2} & \frac{k_3^* T_s}{k_2} & \frac{k_3^* T_s}{k_2} & \frac{k_3^*T_s}{k_2 } + \frac{k_3^*}{k_3^+}  & \frac{2 k_3^* T_s}{k_2} \\
0 & 0 & 0 & 0 & 0 & 0 \\
0 & 0 & 0 & 0 & 0 & 0 \\
0 & 0 & 0 & 0 & 0 & 0 \\
0 & 0 & 0 & 0 & 0 & 0
\end{bmatrix}
\]
We have
\[
\rho(FV^{-1}) = \frac{k_3^* T_s}{k_2}
\]
It follows by Theorem~2.4 that $x^*$ is locally asymptotically stable for any stoichiometric compatibility class satisfying $T_s < k_2/k_3^*$ and unstable for any compatibility class satisfying $T_s > k_2/k_3^*$.
\end{proof}

\bibliographystyle{unsrt}
\bibliography{bib}

\end{document}